\documentclass[a4paper,11pt,reqno]{amsart}

\usepackage{t1enc}
\usepackage{graphicx}
\usepackage{amsthm}
\usepackage{stmaryrd}
\usepackage{color}
\usepackage{fourier}
\usepackage{array}
\usepackage{enumerate}
 \usepackage[all]{xypic}
\usepackage[mathscr]{eucal} 
\usepackage[UKenglish]{babel}
\usepackage[small,bf]{caption}
\usepackage[vcentermath]{youngtab}
\usepackage{tikz}
\usepackage{stackrel}

\newsavebox\MesssBox

\newcommand{\boxenlen}[2]{\savebox{\MesssBox}{\includegraphics[height=#2cm]{#1.pdf} }
\parbox{\wd\MesssBox}{\includegraphics[height=#2cm]{#1.pdf}}}

\newcommand{\boxenlenb}[2]{\savebox{\MesssBox}{\includegraphics[height=#2ex]{#1.pdf} }
\parbox{\wd\MesssBox}{\includegraphics[height=#2ex]{#1.pdf}}}

\newcommand{\spiegely}[1]{\reflectbox{#1}}
\newcommand{\spiegelx}[1]{\rotatebox[origin=c]{-90}{\reflectbox{\rotatebox[origin=c]{90}{#1}}}}
\newcommand{\drehung}[1]{\rotatebox[origin=c]{180}{#1}}

\newcommand{\boxenleny}[2]{\spiegely{\boxenlen{#1}{#2}} }

\newcommand{\boxenlenbd}[2]{\drehung{\boxenlenb{#1}{#2}} }

\newcommand{\boxenlenby}[2]{\reflectbox{\boxenlenb{#1}{#2}} }
\newcommand{\boxenlenbx}[2]{\spiegelx{\boxenlenb{#1}{#2}} }

\newcommand{\comult}{\boxenlenb{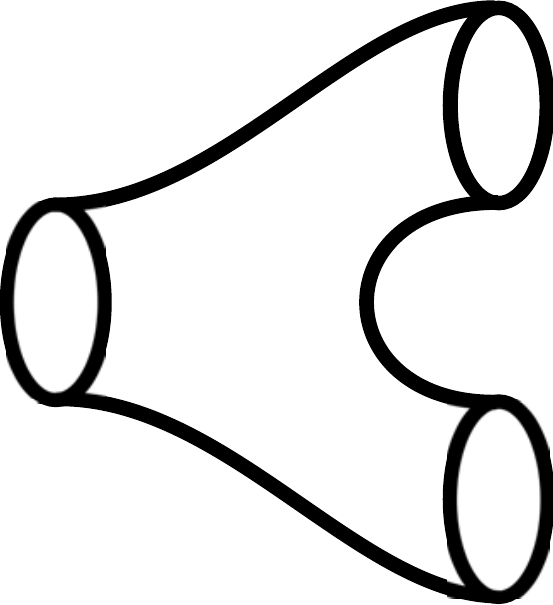}{4}}
\newcommand{\mult}{\boxenlenb{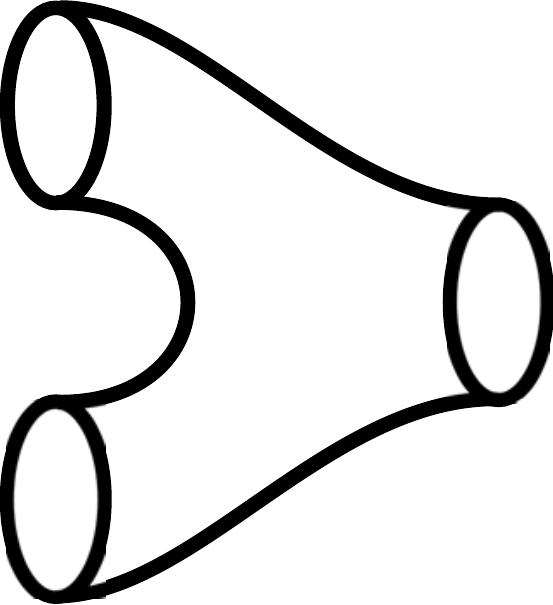}{4}}

\newcommand{\cobid}{\boxenlenb{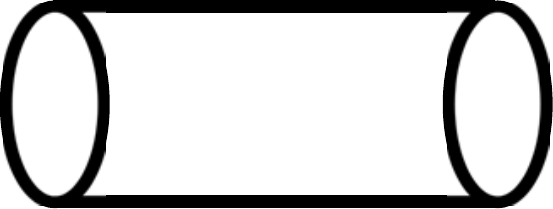}{2}}
\newcommand{\twist}{\boxenlenb{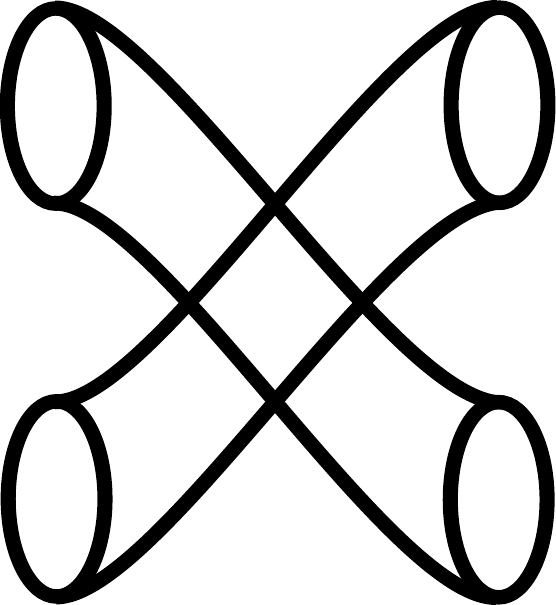}{4}}
\newcommand{\unit}{\boxenlenb{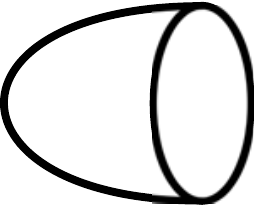}{3}}
\newcommand{\counit}{\boxenlenb{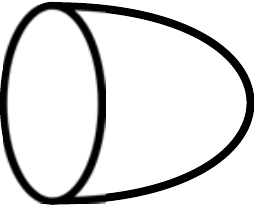}{3}}

\makeatletter
\renewcommand{\big}{\bBigg@{1}}
\renewcommand{\Big}{\bBigg@{1.5}}
\renewcommand{\bigg}{\bBigg@{2.4}}
\renewcommand{\Bigg}{\bBigg@{3.2}}
\newcommand{\bigo}{\bBigg@{1.2}}
\newcommand{\bigc}{\bBigg@{0.7}}
\makeatother

\usepackage[margin=3cm,footskip=25pt,headheight=20pt]{geometry} %von S

\DeclareMathOperator{\id}{id}

 \theoremstyle{plain}
\newtheorem{satz}{Theorem}[section]
\newtheorem*{satz*}{Theorem}
\newtheorem{lemma}[satz]{Lemma}
\newtheorem{prop}[satz]{Proposition}
\newtheorem{kor}[satz]{Corollary}

\theoremstyle{definition}
\newtheorem{defi}[satz]{Definition}
 \newtheorem{bem}[satz]{Remark}
 
 \newtheorem{bsp}[satz]{Example}
 
 \newenvironment{bew}{\begin{proof}}{\end{proof}}

\numberwithin{equation}{section}

\newcommand{\C}{\mathbb{C}}

\newcommand{\vp}{\varphi}
\newcommand{\im}{\,\operatorname{im}\,}
\newcommand{\st}[1]{\EuScript{Y}_{#1}}	%attracting set in springer fiber
\newcommand{\stb}[1]{\widetilde{\EuScript{Y}}_{#1}} %attracting set in spaltenstein var
\newcommand{\St}[1]{S(#1)}
\newcommand{\com}[1]{Y_{#1}} %componente in springer fiber
\newcommand{\comb}[1]{\widetilde{Y}_{#1}} %componente in spaltenstein var
\newcommand{\fx}[1]{\mathcal{F}_\bullet({#1})} %fixed points
\newcommand{\tP}{\EuScript{P}}
\newcommand{\tQ}{\EuScript{Q}}
\newcommand{\up}{\wedge}
\newcommand{\down}{\vee}
\newcommand{\op}[1]{\operatorname{#1}}

\newcommand{\tS}{\mathcal{S}}
\newcommand{\hooklongrightarrow}{\lhook\joinrel\longrightarrow}

\newcommand{\CP}{\mathbb{CP}}
\newcommand{\pt}{pt}
\newcommand{\depG}[1]{\text{\emph{depG}}({#1})}
\newcommand{\ColCob}{\text{\emph{ColCob}}}
\newcommand{\Cob}{\text{\emph{Cob}}}

\newcommand{\BIGOP}[1]{\mathop{\mathchoice%
{\raise-0.22em\hbox{\huge $#1$}}%
{\raise-0.05em\hbox{\Large $#1$}}{\hbox{\large $#1$}}{#1}}}

\author{Gisa Schäfer}
\address{Department of Mathematics, University of Bonn, Endenicher Allee 60, 53115 Bonn, Germany}
\email{gischaef@math.uni-bonn.de}

\begin{document}

\title[2-block Spaltenstein varieties]{A graphical calculus for 2-block Spaltenstein varieties}

\begin{abstract}
We generalise statements known about Springer fibres associated to nilpotents with 2 Jordan blocks to Spaltenstein varieties. We study the geometry of generalised irreducible components (i.e. Bialynicki-Birula cells) and their pairwise intersections. In particular we develop a graphical calculus which encodes their structure as iterated fibre bundles with $\CP^1$ as base spaces and compute their cohomology. At the end we present a connection with coloured cobordisms generalising a construction of Khovanov and Stroppel.
\end{abstract}
\maketitle

\section{Introduction}%%%%%%%%%%%%%%%%%%%%%%%%%%%%%%%%%%%%%%%%%%%%%%%%%%%%%%%%%%%%%%%%%%%%%%%%%%%%%%%%%%%%%%%

For a given nilpotent endomorphism $N$ of $\mathbb{C}^n$, the Springer fibre is the subvariety of the variety of full flags given by the flags fixed under $N$. If we take partial flags instead of full ones, we get Spaltenstein varieties. 

In 1976, Spaltenstein showed that the irreducible components of the Springer fibre are in bijective correspondence with standard tableaux of shape given by the sizes of the Jordan blocks of $N$ \cite{spal}. He then deduced for the Spaltenstein varieties a bijection between its irreducible components and a certain subset of the standard tableaux. 

In general, the geometry of the irreducible components is not well understood. In 2003, Fung considered two special classes of Springer fibres \cite{fung}, including the one where the endomorphism $N$ has at most 2 blocks. 
In this case, he gave an explicit description of irreducible components of the Springer fibre and showed that they are iterated fibre bundles with $\CP^1$ as base spaces. In addition, he used cup diagrams to describe the structure of the irreducible components. 

In \cite{stroppel}, Stroppel and Webster expanded the use of cup diagrams for the description of components of 2-block Springer fibres. Moreover, they introduced what we call generalised irreducible components for Springer fibres. Generalised components are the closure of fixed point attracting cells for a certain torus action (namely the 2-dimensional torus of all diagonal matrices commuting with $N$). The set of generalised components contains the set of irreducible components. They computed the cohomology of generalised components and their pairwise intersections and showed that the intersections are iterated fibre bundles, in particular smooth. Furthermore, they defined a non-commutative convolution algebra structure on the direct sum of cohomologies $\bigoplus_{(w,w')} H^*(\st{w} \cap \st{w'})$ over all pairs of irreducible components or attracting cells, respectively. They used the cup diagram calculus to recover geometrically (using irreducible components) Khovanov's arc algebra \cite{khovcat} which was used in the original construction of Khovanov homology and (using generalised irreducible components) a slightly larger algebra appearing naturally in the Lie theoretic version of Khovanov homology \cite{stroppel2}. 

In this article we consider the special case of 2-block Spaltenstein varieties and generalise the theorems already known in the Springer fibre case. We develop a diagram calculus (dependence graphs and generalised cup diagrams) which describes the geometry in this case and enables us to compute diagrammatically the spaces $\bigoplus_{(w,w')} H^*(\stb{w} \cap \stb{w'})$. To describe an associative non-commutative algebra structure we have to develop the ideas of Khovanov and Stroppel further by introducing the notion of coloured cobordisms (Definition \ref{ColCob}) and realising $\bigoplus_{(w,w')} H^*(\stb{w} \cap \stb{w'})$ as the image of a monoidal functor (``coloured'' 2-dimensional TQFT) from the category of coloured cobordisms to vector spaces. As in the original construction of Khovanov, coloured cobordisms can be used to define an algebra structure on $\bigoplus_{(w,w')} H^*(\stb{w} \cap \stb{w'})$. 

We first use results from Spaltenstein's paper to get a bijective correspondence between irreducible components of Spaltenstein varieties and certain standard tableaux in Section \ref{sec:ZurückführungAufSpringerFasern}. Then, in Section \ref{sec:neuerAbsatz} we consider the theorem of Fung which explicitly describes the irreducible components of 2-block Springer fibres and generalise it to Spaltenstein varieties. 

In Section \ref{sec:fixpkte} we define and study generalised irreducible components which results in a description (Theorem \ref{row strict}) of these generalised components similar to \cite[Theorem 15]{stroppel}. After that, we give a bijective morphism from the generalised irreducible components of Spaltenstein varieties to those of certain Springer fibres (Theorem \ref{Abb nach Springer}).

Subsequently, in Section \ref{sec:Abhängigkeitsgraphen} we generalise the cup diagrams appearing in \cite{fung} and \cite{stroppel} by what we call \textit{dependence graphs}. These dependence graphs consist of labelled and coloured arcs. 
They describe the structure of generalised irreducible components $\stb{w}$ of Spaltenstein varieties visually (Theorem \ref{abh(w)}):
{\renewcommand{\thesatz}{\Alph{satz}}
\begin{satz}
$\stb{w}$ consist of all $N$-invariant flags satisfying the conditions of the dependence graph for $w$.
\end{satz} 
We then extend this to pairwise intersections (Corollary \ref{Abh(w,w')}). 

Next, we use coloured circle diagrams as in \cite{stroppel2} to give a condition for the intersection of generalised irreducible components to be empty in Section \ref{sec:CupCapDiagrams}. 
In Section \ref{sec:iteratedFib} we show that generalised irreducible components and non-empty pairwise intersections of those form iterated fibre bundles, giving a proof that uses cup diagrams. From this we compute the cohomology of the generalised irreducible components and their non-empty intersections using a spectral sequence argument in Section \ref{sec:newSection}. 

Finally, in Theorem \ref{toPartI} we combine the above to see that we can calculate the cohomology of intersections of the generalised irreducible components $\stb w, \stb {w'}$ by counting the number of circles of a certain colour in the circle diagram $CC(w,w')$ associated to the corresponding pairs of row strict tableaux:
\begin{satz}\label{B}
	The following diagram commutes:
	\[\xymatrix{
	\text{pairs of row strict tableaux}  \ar[rrrr]_{\quad \quad \quad (w,w') \mapsto F(CC(w,w'))} \ar[dd] &&&&  \text{vector spaces} \ar[dd]^{\id}\\
	  \\
	 \text{pairs of }\stb w \ar[rrrr]_{(\stb w, \stb {w'}) \mapsto H^*(\stb w \cap \stb {w'}) } &&&&  \text{vector spaces}\\
	 }\]
\end{satz}

In the last section we bring the theorem above in connection with coloured cobordisms and show in Theorem \ref{algebra}:
\begin{satz} \label{C}
There is an associative algebra structure on $\bigoplus_{w,w'}H^*(\stb w \cap \stb{w'})$ given by coloured cobordisms. 
\end{satz}
}

\section*{Acknowledgements} 
I would like to thank my advisor Prof.\ Catharina Stroppel for her encouraging support and helpful advice. I am grateful to the referee for the useful comments. 

%%%%%%%%%%%%%%%%%%%%%%%%%%%%%%%%%%%%%%%%%%%%%%%%%%%%%%%%%%%%%%%%%%%%%%%%%%%%%%%%%%%%%%%%%%%%%%%%%%%%%%%%%%%%%%55

\section{Spaltenstein varieties and first properties}
\label{sec:SpaltensteinVarietätenUndErsteEigenschaften}

We fix integers $n$ and $k$ with $n \geq 2k \geq 0$. 
Let $V$ be an $n$-dimensional complex vector space and let $N: V \to V$ be a nilpotent endomorphism of Jordan type $(n-k,k)$.
Explicitly, we equip $V$ with an ordered basis $\left\{e_1, \dots e_{n-k}, f_1, \dots f_k\right\}$ with the action of $N$ defined by $N(e_i) = e_{i-1}, N(f_i) = f_{i-1}$, where by convention, $e_0 = f_0 = 0.$

\begin{defi}
A \textit{partial flag of type $(i_1,\dots, i_m)$} (where $0<i_1 < \dots < i_m=n$) consists of subspaces $F_{i_l}$ of $V$ with $\dim F_{i_l} = i_l$ and $F_{i_1} \subset F_{i_2} \subset \dots \subset F_{i_m}$. The partial flags of type $(i_1,\dots, i_m)$ form a complex algebraic variety which we denote by $\op{Fl}(i_1, \dots, i_m)$. 

A partial flag is called \textit{$N$-invariant},
if $N F_{i_l} \subset F_{i_{l-1}}$ holds for all $l=1,\dots, m$, where $F_{i_0} := \{0\}$.

The variety of $N$-invariant partial flags of type $(i_1,\dots, i_m)$ is called \textit{Spaltenstein variety of type $(i_1,\dots, i_m)$} and denoted $\op{Sp}(i_1, \dots, i_m)$. 
\end{defi}

\begin{lemma} \label{dim}
For an $N$-invariant flag of type $(i_1,\dots, i_m)$ we have 
	\[\dim F_{i_l} \leq \dim F_{i_{l-1}} +2,
\]
hence $i_{l}-i_{l-1} \leq 2$.
\end{lemma}
\begin{bew}
This follows from the rank-nullity theorem ($\dim W = \dim \ker + \dim \im$) and the fact that $N$ has two Jordan blocks.
\end{bew}

\begin{bem} \label{N-inv}
Let $F_{\bullet} = (F_{j_1} \subset \dots \subset F_{j_r})$ be an $N$-invariant flag. Let $\{j_1, \dots, j_r = n\} \subset \{i_1, \dots , i_m\}$. Then $F'_{\bullet} = (F'_{i_1} \subset \dots \subset F'_{i_m})$ with $F'_s = F_s$ if $s = j_l = i_{l'}$ and $F'_s$ arbitrary otherwise is an $N$-invariant flag as well. 
\end{bem}

\begin{defi}
Let $(i_1, \dots, i_m) \in \mathbb{N}^m$ with $0<i_1 < \dots < i_m=n$. 
A \textit{tableau of shape $(n-k,k)$ of type $(i_1,\dots, i_m)$} is a Young diagram of shape $(n-k,k)$ filled with $(i_l - i_{l-1})$-times the entry $i_l$ for $l = 1, \dots, m$, where $i_0:=0$. 

In the following all tableaux will be of shape $(n-k,k)$ unless stated otherwise.

A \textit{row strict tableau of type $(i_1,\dots, i_m)$} is a tableau of type $(i_1,\dots, i_m)$ with strictly decreasing entries in the rows.

A \textit{standard tableau of type $(i_1,\dots, i_m)$} is a row strict tableau of type $(i_1,\dots, i_m)$ with decreasing entries in the columns.
\end{defi}

\begin{bem}
This is similar to the usual definition of semi-standard tableaux. However, we work with (strictly) decreasing instead of (strictly) increasing rows respectively columns. 
\end{bem}

\begin{bsp} Here is an example for a row strict tableau $w$ and a standard tableau $S$, respectively, of type $(1,3,4,5)$ and shape $(n-k,k)$ for $n=5, k=2$:
\[w = \young(431,53) \hspace{4,5cm} S=\young(531,43)
\]
\end{bsp}

\begin{bem} Note that because of the strictly decreasing rows every number appears at most twice in a row strict tableau and thus also in a standard tableau since we only have two rows. Consequently, we get $i_{l+1}-i_l \leq 2$. Note that this is the property that was proven in Lemma \ref{dim} for indexing set of the Spaltenstein variety. 
\end{bem}

\section{Reduction to Springer fibres}
\label{sec:ZurückführungAufSpringerFasern}

\begin{defi}
A Spaltenstein variety of type $(1,\dots, n)$ is called \textit{Springer fibre}.
\end{defi}

In the following, let $Y$ be the Springer fibre and let $\tS$ be the set of all standard tableaux of type $(1, \dots, n)$ and shape $(n-k,k)$.

In \cite{spal} Spaltenstein constructed a map $\pi: Y \to \tS$ (see also Vargas in \cite{vargas}) and showed:

\begin{satz} {(Spaltenstein-Vargas)} \label{1-1spal} ~\\
The set $\tS$ of standard tableaux of type $(1, \dots , n)$ is in natural bijection with the irreducible components of the Springer fibre $Y$ via $\sigma \mapsto \overline{\pi^{-1}(\sigma)}=:Y_{\sigma}$.
\end{satz}

This theorem holds even if $N$ of arbitrary Jordan type. In general, it is complicated to calculate the closure $\overline{\pi^{-1}(\sigma)}$. But in our special case, where we only have two Jordan blocks, Fung explicitly determined how the irreducible components associated with a standard tableau look like: 
\begin{satz} {\rm(\cite[Theorem 5.2]{fung})}\\
Let $N$ be a nilpotent map of
Jordan type $(n-k,k)$, and let $\sigma$ be a standard tableau of shape $(n-k,k)$. Then the component $Y_{\sigma}$ of
the Springer fibre $Y$ consists of all flags whose
subspaces satisfy the following conditions:
\begin{itemize}
	\item for each $i$ \[ F_i \subset N^{-1}(F_{i-1}) 
\]
\item if $i$ is on the top row of the tableau $\sigma$ and $i-1$ is on the
bottom row, then
\[ F_i= N^{-1}(F_{i-2})
\]
\item if $i$ and $i-1$ are both in the top row of $\sigma$, then
\begin{itemize}
	\item if $F_{i-1} =
N^{-d}(F_{r})$ where $r$ is on the bottom row, then
\[ F_i = N^{-d-1}(F_{r-1}) \]
 \item if $F_{i-1} = N^{-d}(\im N^{n-k-j})$ where $0 \leq j < n-2k$ \footnote[1]{There is a typing error in Fung's paper:
If we have $k < a \leq n-k$ and then replace $a$ by $n-k-j$, we get for $j$ the inequality $n-2k > j \geq 0$ and not $k > j \geq 0$.}
, then
\[ F_i= N^{-d}(\im N^{n-k-j-1}).
\]
\end{itemize}
\end{itemize}
(Here $0$ is thought of being in the top row, $\{0\}=F_0 = \im N^{n-k}$)
\end{satz}

We want to generalise Fung's theorem to Spaltenstein varieties. For this purpose, we first generalise Spaltenstein's theorem.

\begin{defi}
Let $\widetilde{Y}$ be a Spaltenstein variety of type $(i_1, \dots, i_m)$. Let $F_{\bullet} = (F_{i_1} \subset \dots \subset F_{i_m}) \in \widetilde{Y}$. We call the set 
	\[X = X(F) := \left\{ F'_{\bullet} = (F'_1 \subset F'_2 \subset \dots \subset F'_n) : \dim F'_i = i, F'_{i_l} = F_{i_l} \forall l = 1, \dots, m \right\}
\]
the set of the \textit{full flags associated with the partial flag} $F$.
\end{defi}

\begin{defi} (\cite[p. 455]{spal}) \\
Let $I \subset \{1, \dots, n-1\}$.
Then a \textit{subspace of type} $I$ is a set of flags in the flag variety of the form $$\left\{(F_1 \subset \dots \subset F_n) : F_j \text{ is fixed for all } j \in \{1,\dots, n\} \setminus I \right\}.$$

We call $Z$ an $I$\textit{-variety} if it is a union of subspaces of type $I$.
\end{defi}

\begin{bsp} \label{I-var}
\begin{enumerate}[a)]
	\item Let $F_{\bullet}$ be a partial flag in the Spaltenstein variety $\widetilde{Y}$ of type $(i_1, \dots, i_m)$. Then the set $X$ of full flags associated with the partial flag $F_{\bullet}$ is a subspace of type $I =  \{1, \dots, n-1\} \setminus \{i_1, \dots, i_m \}$ in the Springer fibre $Y$.
	\item For $\sigma$ a standard tableau let $I_{\sigma} = \{i |\sigma_i \leq \sigma_{i+1} \}$ where $\sigma_i$ is the number of the column of $\sigma$ in which we find $i$. Spaltenstein showed in \cite[S. 455]{spal} that $Y_{\sigma}$ is an $I_{\sigma}$-variety and $I_{\sigma}$ is maximal with respect to inclusion with this property.
\end{enumerate}
\end{bsp}

\begin{lemma} \label{Isubspace} 
Let $Y$ be the Springer fibre.
\begin{enumerate} [a)]
\item \label{teil1} Let $U$ be a subspace of type $I =  \{1, \dots, n-1\} \setminus \{j_1, \dots, j_r \}$ in the Springer fibre $Y$. Then we have $N F_{j_l} \subset F_{j_{l-1}}$ for every $F_{\bullet} = (F_1, \dots, F_n) \in U$.
\item \label{teil2} Let $U$ be a subspace of type $I$ in $Y$. Then there are no consecutive numbers in $I$.
	\end{enumerate}
\end{lemma}
\begin{bew} 
Let $a > b \in \{1, \dots n \} \setminus I$ and assume all intermediate numbers are in $I$. 
	Then for all $F_{\bullet} = (F_1 \subset \dots \subset F_n) \in U$ we have that $F_a$ and $F_b$ are fixed and all possibilities for $F_{a-1}, \dots, F_{b+1}$ with $F_a \supset F_{a-1} \supset \dots \supset F_{b+1} \supset F_b$ appear. Since $F \in U \subset Y$, we have $N F_a \subset F_{a-1}$ for all possible choices of $F_{a-1}$, thus $N F_a$ lies in the intersection of all $F_{a-1}$. Since the subsets between $F_b$ and $F_{a-1}$ including $F_{a-1}$ are all not fixed, we have 
	\[F_b = \bigcap_{F_{a-1}: F_a \supset F_{a-1} \supset \dots  \supset F_b} F_{a-1} \supset N F_a.
\]%

Consequently, we get \ref{teil1}).
On the other hand, we conclude	 
\begin{align} \label{ungl}
	\dim F_b \geq \dim N F_a \geq \dim F_a -2,
\end{align}
where the last inequality holds, as in Lemma \ref{dim}, because $N$ is of Jordan type $(n-k,k)$. 

If we have $a > b+2$, we get $\dim F_a > \dim F_b +2$, which contradicts \eqref{ungl}. Thus, \ref{teil2}) holds.
\end{bew}

\begin{lemma}\label{spallemma}{\rm (\cite[p. 455]{spal})} ~\\
Let $Y$ be the Springer fibre.
Any subspace of type $I$ contained in $Y$ is contained in an irreducible component which is an $I$-variety.
\end{lemma} 

The following theorem is stated in \cite{spal} in slightly different notation and without proof, so we recall it here.

\begin{satz} \label{Spaltenstein} Let $\widetilde{Y}$ be a Spaltenstein variety of type $(i_1,\dots, i_m)$. Let $I$ $:=\{1, \dots, n-1\} \setminus \{i_1, \dots, i_m \}$, let $\tS$  be the set of standard tableaux of type $(1,2,\dots, n)$ and $I_{\sigma} = \{i |\sigma_i \leq \sigma_{i+1} \}$ for $\sigma \in \tS$, where $\sigma_i$ is the column number of $\sigma$ containing i. Then there is a canonical bijection
\begin{align*}
	\left\{\text{irreducible components of }\widetilde{Y}\right\} \stackrel{1:1}{\longleftrightarrow}\tS_I := \{\sigma \in \tS | I \subset I_{\sigma} \}.
\end{align*}
\end{satz}
\begin{bew} 
\begin{itemize}
	\item The map
\begin{eqnarray*}
  \op{pr}: \quad Z:= \bigcup_{I \subset I_{\sigma}} \com{\sigma}&\to& \widetilde{Y} \\
	(F_1 \subset \dots \subset F_n) &\mapsto& (F_{i_1} \subset \dots \subset F_{i_m})
\end{eqnarray*}
 given by forgetting the subsets of the full flag with indices in $I$ is well-defined:\\
Let $F_{\bullet}=(F_1 \subset \dots \subset F_n) \in Z$, so $F_{\bullet} \in Y_{\sigma}$ for a $\sigma$ with $I \subset I_{\sigma}$. By Example \ref{I-var} b) $Y_{\sigma}$ is an $I_{\sigma}$-variety, thus $F_{\bullet}$ is contained in a subspace of type $I_{\sigma}$. By Lemma \ref{Isubspace} \ref{teil1}) we have $N F_{j_l} \subset F_{j_{l-1}}$ for $l = 1, \dots, r$ and $I_{\sigma} = \{1, \dots, n-1 \} \setminus \{j_1, \dots , j_r \}$. Because of $I \subset I_{\sigma}$ we have $\{j_1, \dots, j_r\} \subset \{i_1, \dots, i_m \}$, and by Remark \ref{N-inv}, $(F_{i_1} \subset \dots \subset F_{i_m}) = \op{pr}(F)$ is an $N$-invariant flag.

\item $\op{pr}$ is surjective:\\
Let $F_{\bullet} \in \widetilde{Y}$ and let $X$ be the set of associated full flags. By Example \ref{I-var} a) $X$ is a subspace of type $I$ in $Y$. Hence, by Lemma \ref{spallemma} there exists an irreducible component $Y_{\sigma}$ with $X \subset \com{\sigma}$, where $\com{\sigma}$ is an $I$-variety. Therefore, we have $ I \subset I_{\sigma}$ because of the maximality of $I_{\sigma}$. So we have 
	\[ F \in \op{pr}(X) \subset \op{pr}(\com{\sigma}) \subset \op{pr}\left(\bigcup_{I \subset I_{\sigma}} \com{\sigma}\right).
\]
\item Using \cite[21.1]{hum} we get that $\op{pr}$ is a morphism of varieties mapping $Y_\sigma$ with $I \subset I_{\sigma}$ to an irreducible component $\op{pr}(Y_{\sigma})$ of $\widetilde{Y}$. Since $Y_{\sigma}$ is an $I$-variety, we have $\op{pr}^{-1}(\op{pr}(Y_{\sigma})) = Y_{\sigma}$, thus the $\op{pr}(Y_{\sigma})$ are distinct.
\end{itemize}
Altogether, there is a 1-1 correspondence between the irreducible components of $\widetilde{Y}$ and the set $\{Y_{\sigma} | I \subset I_{\sigma} \}$, thus also between the irreducible components of $\widetilde{Y}$ and the set $\{ \sigma \in \tS | I \subset I_{\sigma} \}= \tS_I$.
\end{bew}

The bijection from Theorem \ref{Spaltenstein} will be made explicit in the next section.

\section{Explicit description of irreducible components}
\label{sec:neuerAbsatz}

\begin{defi} \label{phi} Let $S$ be a standard tableau of type $(i_1,\dots, i_m)$. We associate  a standard tableau $S'$ of type $(1, \dots, n)$ with $S$ as follows: By definition, in a standard tableau of type $(i_1, \dots i_m)$ there are at most two entries $i_j$ for all $j$. If there are two entries $i_j$, then they have to be in different rows. Hence, we can associate a unique standard tableau of type $(1, \dots, n)$ with a standard tableau of type $(i_1, \dots, i_m)$ by changing the entry $i_j$ in the lower row to $i_j-1$ if there are two such entries in the tableau. This is possible because if $i_j$ is a double entry, then there is no $i_j-1$ in the tableau.

In this way we get an injective map
	\begin{align*}\vp: \quad &\{\text{standard tableaux of type }(i_1,\dots, i_m)\} \\
	&\quad \hooklongrightarrow \{\text{standard tableaux of type } (1, \dots, n) \}. 
\end{align*}
\end{defi}

\begin{bsp} For $n=8$ and $k = 3$ the standard tableau $\young(86431,763)$ of type $(1,3,4,6,7,8)$ is mapped to $\young(86431,752)$ by $\vp$.
\end{bsp}

To summarise, we obtain a bijection between irreducible components and standard tableaux of type $(i_1, \dots, i_m)$ as follows:
Every irreducible component of $\widetilde{Y}$ is the image of an irreducible component $Y_{\sigma}$ via $\op{pr}$ for $\sigma \in \tS_I$. To $Y_{\sigma}$ we assign a standard tableau $\sigma$ of type $(1, \dots, n)$ via the Spaltenstein-Vargas bijection from Theorem \ref{1-1spal}. As shown in the following theorem, $\sigma$ is in the image of $\vp$ and thus corresponds to a standard tableau $S$ of type $(i_1, \dots, i_m)$.

\begin{satz} \label{1-1}
The irreducible components of the Spaltenstein variety of type $(i_1,\dots, i_m)$ are in natural bijection with the standard tableaux of type $(i_1,\dots, i_m)$ and shape $(n-k,k)$.
\end{satz}
\begin{bew}
We have 
\begin{align*}\tS_I &= \{\sigma \in \tS | \sigma_i \leq \sigma_{i+1} \forall i \in I\}\\
&= \{\sigma \in \tS | i  \text{ occurs in the bottom and  } i+1 \text{ in the top row of } \sigma \text{ resp. }  \forall i \in I \}.
\end{align*}
Since $I = \{1, \dots, n-1\} \setminus \{i_1, \dots, i_m \}$, we get $i \in I$ if and only if $i=i_j-1$ for $i_j$ a double entry.
Thus, $\tS_I = \im \vp$ and the theorem follows by the comments above.
\end{bew}

\begin{defi}
We denote by $\widetilde{Y}_S$ the irreducible component of the Spaltenstein variety $\widetilde{Y}$ corresponding to a standard tableau $S$ of type $(i_1, \dots, i_m)$ by Theorem \ref{1-1}.
\end{defi}

\begin{bem}\label{zurückführung}
From the proofs of Theorem \ref{Spaltenstein} and Theorem \ref{1-1} we particularly get the following:
If $F_\bullet \in \comb{S}$, then all full flags associated with $F_\bullet$ are in $\com{\vp(S)}$, i.e. $\op{pr}^{-1}(F_\bullet) \subset \com{\vp(S)}$. On the other hand, for a full flag $F'_\bullet$ with $F'_\bullet \in \com{\sigma}$ such that $I \subset I_{\sigma}$ we know that the projected partial flag lies in $\comb{\vp^{-1}(\sigma)}$.
\end{bem}

\begin{bem} \label{doppelte} 
For $i_l$ is a double entry in $S$ we have $F_{i_l} = N^{-1} F_{i_{l-1}}$ because of the rank-nullity theorem. 
\end{bem}

We formulate as a first result a generalisation of \cite[Theorem 5.2]{fung}:
\begin{satz}[Explicit description of irreducible components] \label{Komponenten} ~\\
Let $S$ be a standard tableau of type $(i_1,\dots, i_m)$. Then the irreducible component $\comb{S}$ of the Spaltenstein variety $\widetilde{Y}$ of type $(i_1,\dots, i_m)$ consists of all flags, whose subspaces satisfy the following conditions:

\begin{itemize}
	\item for each $l$	\[F_{i_l} \subset N^{-1}(F_{i_{l-1}})
\] 
\item if $i_l$ is on the top row of the tableau $S$ and $i_l-1$ is in the bottom row and $i_l-1$ is not a double entry, 
then
	\[F_{i_l} = N^{-1}(F_{i_l-2})
\]
\item if $i_l$ and $i_l-1$ are both in the top row of $S$, then
\begin{itemize}
	\item if $F_{i_l-1} = N^{-d}(F_r)$ where $r$ is in the bottom row and not a double entry, then
	\[F_{i_l} = N^{-d-1}(F_{r-1})
\]
\item if $F_{i_l-1} = N^{-d}(\im N^{n-k-j})$ where $0 \leq j < n-2k$, then 
	\[F_{i_l} = N^{-d}(\im N^{n-k-j-1}).
\]
\end{itemize}
\end{itemize}
(Here $0$ is thought of being in the top row, $\{0\}=F_0 = \im N^{n-k}$)
\end{satz}

\begin{bew} [Proof of Theorem \ref{Komponenten}] 
Let $F_\bullet \in \comb S$ and let $\widehat{F}_{\bullet} \in X(F_{\bullet})$ be an associated full flag. By Remark \ref{zurückführung} we have $\widehat{F}_{\bullet} \in \com{\vp(S)}$, so $\widehat{F}_{\bullet}$ meets the conditions of \cite[Theorem 5.2]{fung}. By comparing the different cases and taking Remark \ref{doppelte} into account we see that only the above conditions remain.

Conversely, let $F_\bullet \in \widetilde{Y}$ be a partial flag satisfying the conditions above. Let  $\widehat{F}_\bullet$ be a full flag associated with $F_\bullet$ . 
By considering all possible cases, we see that the conditions of \cite[Theorem 5.2]{fung} are satisfied with respect to $\vp(S)$, thus $\widehat{F}_\bullet \in \com{\vp(S)}$, so by Remark \ref{zurückführung} we have $F_\bullet = \op{pr}(\widehat{F}_\bullet) \in \comb{S}$.
\end{bew}

\begin{defi}
If for all flags satisfying the conditions the subset $F_{i_l}$ is specified as $F_{i_l} = N^{-j}(F_{i_s})$ for some $j > 0$ or $F_{i_l} = N^{-j}(\im N^t)$ for some $j \geq 0$, it is called \textit{dependent}. If a subset is not dependent, it is called \textit{independent}.
\end{defi}

\section{Torus fixed points and generalised irreducible components}
\label{sec:fixpkte}

\begin{bem} [Origin of the $\mathbb{C}^*$-action]~\\
Let $T \cong (\C^*)^n$ be the torus of diagonal matrices in the basis given by the $e_i$'s and $f_i$'s. $T$ acts on the partial flag variety $\op{Fl}(i_1, \dots , i_m)$ via its action on the $e_i$'s and $f_i$'s. 

For $t \in T$ acting on the Spaltenstein variety as well, it has to commute with $N$. 
For  $t = \begin{pmatrix}
\lambda_1 & & \\
&\ddots& \\
& & \lambda_n
\end{pmatrix}$ 
we have $N t = tN$ if and only if $\lambda_1 = \dots = \lambda_{n-k}$ and $\lambda_{n-k+1} = \dots = \lambda_n$.
Therefore, the part of $T$ commuting with $N$ is isomorphic to $(\C^*)^2$.

Now we choose the cocharacter 
\begin{align*}
	\C^* &\to (\C^*)^2 \\
	t & \mapsto (t^{-1},t)
\end{align*}
 and get an action of $\C^*$ on $\op{Sp}(i_1, \dots, i_m)$.
\end{bem}

\begin{lemma} \label{fixed points}
For this $\C^*$-action on $\op{Sp}(i_1, \dots, i_m)$ we get a natural bijection
\begin{align*}
\left\{\text{\rm{row strict tableaux of type }} (i_1,\dots, i_m)\right\} &\stackrel[\Psi]{1:1}{\longleftrightarrow} \left\{\text{\rm{fixed points of the action}}\right\}\\
w &\mapsto \fx{w}, 	
\end{align*} 
where $\fx{w}$ the partial flag with $\mathcal{F}_{i_l}(w) = \left\langle\{e_j, f_r | j \leq t_{i_l}, r \leq b_{i_l}\}\right\rangle$, where $t_s$ is the number of indices smaller than or equal to $s$ in the top row and similarly for $b_s$ in the bottom row.
\end{lemma}
\begin{bew} 
The $\C^*$-action is explicitly given by $t.e_i = t^{-1} e_i, t.f_i = t f_i$.

By writing an element of $\mathcal{F}_{i_l}(w)$ in the basis one can directly see that $\fx{w}$ is a fixed point under the action.
By construction we have $N \mathcal{F}_{i_l}(w) \subset \mathcal{F}_{i_{l-1}}(w)$, so $\fx{w} \in \widetilde{Y}$.

$\Psi$ is injective, since $\fx{w} = \fx{w'}$ implies $\mathcal{F}_{i_l}(w) = \mathcal{F}_{i_l}(w')$ and hence $t_{i_l} = t'_{i_l}$ and $b_{i_l} = b'_{i_l}$. Inductively we get $w = w'$. 

The surjectivity follows by inductively showing that because of being a fixed point each $F_{i_l}$ has to be of the form $\left\langle e_1, \dots, e_r, f_1, \dots, f_s \right\rangle$ for some $r, s$ and constructing a associated row strict tableau by putting a number in the top or in the bottom row depending on whether a $e_r$ or a $f_s$ was added to construct $F_{i_l}$ out of $F_{i_{l-1}}$. 
\end{bew}

\begin{defi} \label{weightseq} Let $w$ be a row strict tableau of type $(i_1,\dots, i_m)$. Let $w_{\down}$ be the set of numbers in the bottom row of the tableau, $w_{\up}$ the set of numbers in the top row and $w_\times$ the set of double entries.\\
We consider the sequence ${\bf
a}=a_1 a_2 a_3\ldots a_n$, where $a_{i-1} = a_i = \times$ if $i \in w_\times$ and otherwise $a_i=\wedge$ if $i\in w_\wedge$ and $a_i=\vee$
if $i\in w_\vee$, and call it the {\it weight sequence} of $w$.

Associate with $w$ a cup diagram $C(w)$ as follows: We consider the weight sequence and build the diagram
inductively by adding an arc between any adjacent pair $\vee\wedge$ (ignoring all $\times$'s),
and then continuing the process for the sequence ignoring the already connected pairs. After that we match all the remaining adjacent $\wedge\vee$-pairs and again ignore all $\times$'s and the alreaddy connected points.  

Now, several row strict tableaux of type $(i_1,\dots, i_m)$ have the same cup diagram. Among all the row strict tableaux which have the same cup diagram as $w$ there is one standard tableau. This standard tableau can be constructed by putting every left endpoint of an arc in the cup diagram in the bottom row, every right endpoint or unmatched point in the top row and then inserting the double entries. Call this $S(w)$. 
\end{defi}

\begin{bsp}
	\[w = \young(531,643) \hspace{4,5cm} S(w)=\young(653,431)
\]
\begin{align*}
	&w_{\up} = \{1, 3, 5\}, \quad w_{\down} = \{3, 4, 6\}, \quad w_{\times} = \{3\} = S(w)_{\times}, \\
	&S(w)_{\up} = \{3, 5, 6\}, \quad S(w)_{\down} = \{1, 3, 4\}
\end{align*}
weight sequence of $w$: $\up \times \times \down \up \down$, \quad  weight sequence of $S(w)$: $\down \times \times \down \up \up$

$$C(w) = \begin{tabular}[c]{c}
\begin{tikzpicture}
		\begin{scope}
		\draw (0,0) node[above=-1.55pt] {$1$};
		\draw (0.4,0) node[above=-1.55pt] {$2$};
		\draw (0.8, 0) node[above=-1.55pt] {$3$};
		\draw (1.2,0) node[above=-1.55pt] {$4$};
		\draw (1.6,0) node[above=-1.55pt] {$5$};
		\draw (2,0) node[above=-1.55pt] {$6$};
	
		\draw (0.8,0) node {$\times$};
		\draw (0.4,0) node {$\times$};
		
		\draw (1.2,0) arc (180:360:0.2);
		\draw (0,0) arc (180:360:1.0);
		\end{scope}
		\end{tikzpicture} 
		\end{tabular}
		= C(S(w))$$		
\end{bsp}

\begin{defi}
Let $P =\left\langle e_1, \dots , e_{n-k}\right\rangle$, $Q = \left\langle f_1, \dots, f_k\right\rangle$ and let $\widetilde{Y}$ be a Spaltenstein variety of type $(i_1, \dots, i_m)$.
For each flag $F_{\bullet}$ in $\widetilde Y$, we can obtain a flag (with no longer necessarily distinct spaces) in $P$ by taking the intersections $\tP_i=F_i\cap P$, and similarly in $Q$ by taking $\tQ_i=\alpha(F_i/(F_i\cap P))$ with $\alpha: V/P \stackrel{\cong}{\rightarrow} Q$.  We can define the new flag $F'_{\bullet}$ by putting $F_i':=\tP_i+\tQ_i\subset P\oplus Q= V$. Let $\stb w^0$ be the subvariety of partial flags $F_{\bullet}$ in $\widetilde{Y}$ with the property that $F'_{\bullet} = \fx w$ holds. Let $\stb w = \overline{\stb w^0}$ be its closure. If $(i_1, \dots, i_m) = (1, \dots, n)$ we write $\st w$ instead of $\stb w$. In the following, we call the $\stb w$ \textit{generalised irreducible components}, even though this is not a standard terminology. But as one can see in the next theorem, the set of generalised irreducible components contains the set of irreducible components. 
\end{defi}

\begin{satz}\label{row strict}
 
 Let $w$ be a row strict tableau of type $(i_1,\dots, i_m)$. Then $\stb w$ is
   the subset of $\comb {\St w}$ containing exactly the flags $F_{\bullet}$ which satisfy the additional property: if $i\in \left(w_{\wedge}\cap \St w_{\vee}\right) \smallsetminus w_{\times}$, then $F_i = \mathcal{F}_i(w)$. 

In particular, for any standard tableau $S$, we have $\stb S=\comb S$.
\end{satz}
\begin{bew}
First we confirm that these relations hold on $\stb w^0$ (and thus on $\stb w$, since they are closed conditions).% 

Consider first the case where $(i_1,\dots, i_m) = (1,\dots, i-1, i+1, \dots, n)$ for some $i$:\\
We use the map $\vp$ from Definition \ref{phi} for row strict tableaux as well. This is possible since in row strict tableaux the double entries also appear in different rows.

Let $F_\bullet \in \stb{w}^0$. Since $(F_1 \subset \dots \subset F_{i-1} \subset F_{i+1} \subset \dots \subset F_n)$ is $N$-invariant, $(F_1 \subset \dots \subset F_{i-1}\subset F_i \subset F_{i+1} \subset \dots \subset F_n)$ is also $N$-invariant for each possible $F_i$.

Hence, by \cite[Theorem 15]{stroppel} we have $(F_1 \subset \dots \subset F_n) \in \com{S(\vp(w))}$ and $F_j = \mathcal{F}_j(\vp(w))$ for all $j \in \vp(w)_{\wedge} \cap \St{\vp(w)}_{\vee} = (w_{\wedge} \cap \St{w}_{\down}) \setminus \{i+1\}$. So by Remark \ref{zurückführung} and $\comb{\vp^{-1}(S(\vp(w)))} = \comb{S(w)}$ we get the relations.

On the other hand, let $F_\bullet \in \comb{S(w)}$ with $F_j = \mathcal{F}_j(w)$ for $j\in \left(w_{\wedge}\cap \St w_{\vee}\right) \smallsetminus \{i+1\}$. 
For $\widehat{F}_\bullet$ a full flag associated with $F_\bullet$ by Remark \ref{zurückführung} and $\St{\vp(w)} = \vp(\St{w})$ we get $\widehat{F}_\bullet \in \com{S(\vp(w))}$. Since $\vp(w)_{\wedge} \cap \St{\vp(w)}_{\vee} =  (w_{\wedge} \cap \St{w}_{\down}) \setminus \{i+1\}$, $\widehat{F}_\bullet$ satisfies the conditions of \cite[Theorem 15]{stroppel}, and $\widehat{F}_\bullet \in \st{\vp(w)}$ follows. In particular, we have $F_\bullet = \op{pr}(\widehat{F}_\bullet) \in \stb{w}$.

For general $(i_1,\dots, i_m)$ the reasoning is analogous, since the proof at most uses properties of $F_{i-1}$ and $F_{i+1}$.
\end{bew}

\begin{defi} \label{p-Abb}
Let $w$ be a row strict tableau of type $(i_1, \dots, i_m)$. Let $I$ be the set $\{1, \dots, n\} \smallsetminus \{i_1, \dots, i_m\}=: \{j_1, \dots, j_r\}$.
We associate with $w$ a row strict tableau of type $(1, \dots, n-2r)$ as follows: 
We delete the boxes with double entries, i.e. those with $j_l+1$, $l=1, \dots, r$. Then we replace the entries $a \in \{j_l+2, \dots, j_{l+1}-1\}$ by $a-2l$ for $l = 1, \dots, r$. The result is still a row strict tableau, which contains the entries $1, \dots, n-2r$ only once.

Thus, we get a map 
\begin{align*}	p: &\left\{  \text{row strict tableaux of type } (i_1, \dots ,i_m) \text{ of shape } (n-k, k) \right\} \\ &\to \left\{ \text{row strict tableaux of type } (1, \dots, n-2r) \text{ of shape } (n-k-r, k-r)\right\}
\end{align*}

For example
	\[p\left( \young(6543,731)\right)= \young(432,51).
\]

We define $\pi: \stb{w} \to \st{p(w)}$ by
\begin{align*}
	(F_1 \subset \dots &\subset F_{j_1-1} \subset F_{j_1+1} \subset \dots \subset F_{j_2-1} \subset \dots \\ 
	&\dots \subset F_{j_l+1} \subset \dots \subset F_{j_{l+1}-1} \subset \dots \subset F_n) \\
	\mapsto (F_1 \subset \dots &\subset F_{j_1-1} \subset N F_{j_1+2} \subset \dots \subset N F_{j_2-1} \subset \dots \\ &\dots \subset N^{l} F_{j_l+2} \subset \dots \subset N^{l} F_{j_{l+1}-1} \subset \dots \subset N^{r} F_n)
\end{align*}
\end{defi}

\begin{satz} \label{Abb nach Springer}
The map $\pi$ from above is an isomorphism of varieties.
\end{satz}
\begin{bew}
Since $p$ as well as $\pi$ are compositions of maps which only forget one index, it is enough to consider $(1, \dots, i-1, i+1, \dots, n)$ with $I=\{i\}$. In this case, the maps $p$ and $\pi$ are of the following form:
\begin{align*}	p: &\left\{  \text{row strict tableaux of type } (1, \dots i-1, i+1, \dots n) \right\} \\ &\to \left\{ \text{row strict tableaux of type } (1, \dots n-2) \text{ of shape } (n-k-1, k-1)\right\}
\end{align*}
is the map which sends a tableau to another one by deleting the boxes with $i+1$ in it and replacing the numbers $i+2, \dots, n$ by $i, \dots, n-2$; and 
\begin{align*}
	\pi: (F_1 \subset \dots \subset F_{i-1} \subset F_{i+1} \subset \dots \subset F_n) \mapsto (F_1 \subset \dots \subset F_{i-1} \subset N F_{i+2} \subset \dots \subset N F_n).
\end{align*}

Since $\dim \ker N|_{F_j} = 2$ for $j \geq i+1$ we get that $(F_1 \subset \dots \subset F_{i-1} \subset N F_{i+2} \subset \dots \subset N F_n)$ is a full flag. The subset relations are clear or follow from Remark \ref{doppelte}.
In addition, for a $N$-invariant flag $(F_1 \subset \dots \subset F_{i-1} \subset F_{i+1} \subset \dots \subset F_n)$ we have that $(F_1 \subset\dots \subset F_{i-1} \subset N F_{i+2} \subset \dots \subset N F_n)$ is $N'$-invariant where $N' = N |_{N V}$.

Let $(F_1 \subset \dots \subset F_{i-1} \subset F_{i+1} \subset \dots \subset F_n) \in \comb{S}$ for $S$ a standard tableau. By considering the different possible cases we deduce from Theorem \ref{Komponenten} that $(F_1 \subset \dots \subset F_{i-1} \subset N F_{i+2} \subset \dots \subset N F_n)$ satisfies the conditions of \cite[Theorem 5.2]{fung}, so $(F_1 \subset \dots \subset F_{i-1} \subset N F_{i+2} \subset \dots \subset N F_n) \in \com{p(S)}$.
Since this holds for every standard tableau, it holds in particular for $S(w)$.

Let $(F_1 \subset \dots \subset F_{i-1} \subset F_{i+1} \subset \dots \subset F_n) \in \stb{w}$ for $w$ a row strict tableau.
From Theorem \ref{row strict} we get $F'_{j-2} = N F_j = N \mathcal{F}_j(w) = \mathcal{F}_{j-2}(p(w))$ for $j \in \left(w_{\wedge}\cap \St w_{\vee}\right) \smallsetminus w_{\times}$. Since $p(w)_{\wedge} \cap S\big(p(w)\big)_{\vee} = p\big(\left(w_{\wedge}\cap \St w_{\vee}\right) \smallsetminus w_{\times}\big)$ the conditions of \cite[Theorem 15]{stroppel} are satisfied and $(F_1 \subset \dots \subset F_{i-1} \subset N F_{i+2} \subset \dots \subset N F_n) \in \st{p(w)}$ follows.

Now we consider the map $\pi': \st{p(w)} \to \stb{w}$ given by
	\[\left(F'_1 \subset \dots \subset F'_{n-2}\right) \mapsto \left(F'_1 \subset \dots \subset F'_{i-1} \subset N^{-1} F'_{i-1} \subset \dots \subset N^{-1} F'_{n-2}\right).
\] 
Analogously to the above calculation one can compute that it is well-defined, and it is an inverse to $\pi$.

By considering $\widetilde{Y}$ as subset of $Gr(1, \C^n) \times \dots \times Gr(i-1, \C^n) \times Gr(i+1, \C^n) \times \dots \times Gr(n,\C^n)$ one can show that $\pi$ is a morphism of varieties. Similarly, $\pi'$ is a morphism of varieties.
\end{bew}

%%%%%%%%%%%%%%%%%%%%%%%%%%%%%%%%%%%%%%%%%%%%%%%%%%%%%%%%%%%%%%%%%%%%%%%%%%%%%%%%%%%%%%%%%%%%%%%%%%%%%%%%%%%%%%%%%
\section{Generalised irreducible components via dependence graphs}
\label{sec:Abhängigkeitsgraphen}

In this section dependence graphs are used to visualise the description of the irreducible components and generalised irreducible components from the last sections. The proofs consist of combinatorial arguments. For more details see \cite{me}.

\subsection{Dependence graphs describing irreducible components}
\label{sec:DependenceGraphsForStandardTableaux}

\begin{defi}
Let $S$ be a standard tableau of type $(i_1,\dots, i_m)$. The \textit{extended cup diagram for} $S$, $eC(S)$, is defined as follows: We expand the weight sequence from Definition \ref{weightseq} by adding $n-2k$ $\down$'s on the left, i.e. ${\bf
a}=\underbrace{\down \dots \down}_{n-2k}a_1 a_2 a_3\ldots a_n$. Then we connect the $\down \up$-pairs as before. If a cup is starting at one of the newly added $\down$'s, we colour it green. (Green lines are represented by dashed lines for the black and white version.)
\end{defi}

\begin{bsp}$n=7, k = 3, n-2k= 1$, \quad
$S = \young(7543,631)$, \quad $\bf{a} = \down \down \times \times \up \up \down \up$ \\
$eC(S) = $\begin{tabular}[c]{c}		\begin{tikzpicture}
		\begin{scope}
		\draw (0,0) node[above=-1.55pt] {$0$};
		\draw (0.4,0) node[above=-1.55pt] {$1$};
		\draw (0.8, 0) node[above=-1.55pt] {$2$};
		\draw (1.2,0) node[above=-1.55pt] {$3$};
		\draw (1.6,0) node[above=-1.55pt] {$4$};
		\draw (2,0) node[above=-1.55pt] {$5$};
		\draw (2.4,0) node[above=-1.55pt] {$6$};
		\draw (2.8,0) node[above=-1.55pt] {$7$};
		\draw (0.8,0) node {$\times$};
		\draw (1.2,0) node {$\times$};
		\draw (2.4,0) arc (180:360:0.2);
		\draw (0.4,0) arc (180:360:0.6);
		\draw[dashed,thick,green!50!black] (0,0) arc (180:360:1.0);
		\end{scope}
		\end{tikzpicture}  \end{tabular}
\end{bsp}

\begin{bsp} \label{gewünschtes}  If $n=2k$, the extended cup diagrams coincide with the cup diagrams, for example 
$n=4, k=2$. Then we have the following two standard tableaux of type $(1,2,3,4)$:\\
$S_1 = \young(43,21)$ and $S_2 = \young(42,31)$. \\
\begin{align*}
	eC(S_1) = \begin{tabular}[c]{c}
		\begin{tikzpicture}
			\begin{scope}
			\draw (0.4,0) node[above=-1.55pt] {$1$};
			\draw (0.8, 0) node[above=-1.55pt] {$2$};
			\draw (1.2,0) node[above=-1.55pt] {$3$};
			\draw (1.6,0) node[above=-1.55pt] {$4$};
			\draw (0.8,0) arc (180:360:0.2);
			\draw (0.4,0) arc (180:360:0.6);
			\end{scope}
			\end{tikzpicture}	\end{tabular} \hspace{2cm}
	eC(S_2) = \begin{tabular}[c]{c}  \begin{tikzpicture}
			\begin{scope}
			\draw (0.4,0) node[above=-1.55pt] {$1$};
			\draw (0.8, 0) node[above=-1.55pt] {$2$};
			\draw (1.2,0) node[above=-1.55pt] {$3$};
			\draw (1.6,0) node[above=-1.55pt] {$4$};
			\draw (0.4,0) arc (180:360:0.2);
			\draw (1.2,0) arc (180:360:0.2);
			\end{scope}
			\end{tikzpicture} \end{tabular}
\end{align*}

\end{bsp}

\begin{defi} \label{defiabh}
Let $S$ be a standard tableau of type $(i_1,\dots, i_m)$.
The \textit{dependence graph} for $S$, $\depG{S}$, is defined as follows: 

We have $m+(n-2k)+1$ given nodes, numbered $-(n-2k)$ to $0$ and $i_1$ to $i_m$. We label the nodes with $F_j$ for $j\in \{i_1,\dots, i_m\}$ and with $\{0\}$ for the node $0$; the remaining nodes are left unlabelled.

If $i_s$ is labelled with $\times$ in the extended cup diagram, then in the dependence graph we connect $i_s -2$ and $i_s$ and label the resulting arc with $N^{-1}$.

Now, if $i<j$ are connected in the extended cup diagram, then in the dependence graph we connect the nodes $i-1$ and $j$ by an arc of the same colour.  
We label the black arcs with $N^{-l}$ for $l = \frac12(j-(i-1))$ and the green ones with $e_l$.
\end{defi}

\begin{bem}
Note that $l$ always is an integer. We constructed the extended cup diagram by connecting adjacent nodes after an even number in between them is deleted. Thus $i$ and $j$ have different parity and $i-1$ and $j$ have the same parity.
\end{bem}

\begin{bsp} \label{bspS}
$S = \young(7543,631)$ \quad \quad
$eC(S) = $	\begin{tabular}[c]{c}	\begin{tikzpicture}
		\begin{scope}
		\draw (0,0) node[above=-1.55pt] {$0$};
		\draw (0.4,0) node[above=-1.55pt] {$1$};
		\draw (0.8, 0) node[above=-1.55pt] {$2$};
		\draw (1.2,0) node[above=-1.55pt] {$3$};
		\draw (1.6,0) node[above=-1.55pt] {$4$};
		\draw (2,0) node[above=-1.55pt] {$5$};
		\draw (2.4,0) node[above=-1.55pt] {$6$};
		\draw (2.8,0) node[above=-1.55pt] {$7$};
		\draw (0.8,0) node {$\times$};
		\draw (1.2,0) node {$\times$};
		\draw (2.4,0) arc (180:360:0.2);
		\draw (0.4,0) arc (180:360:0.6);
		\draw[dashed,thick,green!50!black] (0,0) arc (180:360:1.0);
		\end{scope}
		\end{tikzpicture} \end{tabular}\\
$\depG{S} = $\begin{tabular}[c]{c}
\begin{tikzpicture}
		\begin{scope}
	
		\draw (0,0) node [above=-1.55pt] {};
		\draw (1,0) node[above=-1.55pt] {$\{0\}$};
		\draw (2, 0) node[above=-1.55pt] {$F_1$};
		\draw (3,0) node[above=-1.55pt] {$F_3$};
		\draw (4,0) node[above=-1.55pt] {$F_4$};
		\draw (5,0) node[above=-1.55pt] {$F_5$};
		\draw (6,0) node[above=-1.55pt] {$F_6$};
		\draw (7,0) node[above=-1.55pt] {$F_7$};
		\draw (2,0) arc (180:360:0.5);
	  \draw (5,0) arc (180:360:1.0);
		\draw (1,0) arc (180:360:1.5);
		\draw[dashed,thick,green!50!black] (0,0) arc (180:360:2.5);
		\draw (2.5,-0.7) node {$N^{-1}$};
		\draw (2.5,-1.7) node {$N^{-2}$};
		\draw (2.5,-2.7) node {$e_3$};
			\draw (6,-1.2) node {$N^{-1}$};
		\end{scope}
		\end{tikzpicture}
		\end{tabular}
\end{bsp}

\begin{defi} \label{bogenbez}
Let $B$ be an arc in $\depG{S}$. Then we denote by $s(B)$ the number of the left endpoint of the arc and by $t(B)$ the number of the right endpoint. (Here, by number we mean the number of the node as defined in Definition \ref{defiabh} and not its position.) We define the width $b(B)$ via $b(B) = \frac12(t(B)-s(B))$.

An arc $B'$ \textit{is nested inside} $B$ if we have $s(B) \leq s(B') <t(B') \leq t(B)$. Note that $B$ is nested inside $B$. 
An \textit{arc sequence} from $a$ to $b$ is given by arcs $B_1, \dots, B_r$ with $s(B_1) = a, t(B_r) = b$ and $t(B_i) = s(B_{i+1})$ for $i = 1, \dots, r-1$. 
For $G$ a green arc let $g(G)$ be the number of green arcs nested inside $G$. 
\end{defi}

\begin{bem} \label{check}
In this notation, by the definition of the dependence graph the labelling of an arc $B$ in $\depG{S}$ is given by $N^{-b(B)}$ or $e_{b(B)}$, respectively.
\end{bem}

\begin{prop} \label{bogen}
\begin{enumerate}[a)]
\item \label{schwarzer} Let $B$ be a black arc in the dependence graph with width $b(B)>1$. Then there is a black arc sequence from $s(B)+1$ to $t(B)-1$.
\item \label{grüner} Let $B$ be a green arc in the dependence graph with $b(B)>1$. If there is no green arc nested inside $B$, then there is a black arc sequence from $0$ to $t(B)-1$. If there is a green arc nested inside $B$, then there is a black arc sequence from the rightmost endpoint of the green arcs nested inside $B$ to $t(B)-1$. 
\end{enumerate}
\end{prop}
\begin{bew}
\begin{enumerate}[a)] 
\item By the construction of $B$ from the extended cup diagram we have $t(B) \in S_{\up}$. Since $b(B) > 1$ we get $t(B)-1 \in S_{\up}$. Since by construction the arcs do not intersect, there has to be a black arc $B_1$ nested inside $B$ with $t(B)-1= t(B_1)$ and $s(B_1) > s(B)$. If $B_1$ is not the desired arc sequence, we consider $s(B_1)$. If $s(B_1) \in S_{\down}$ and not a double entry, then there would be an arc $A$ nested between $B$ and $B_1$ which is a contradiction. Thus we have $s(B_1) \in S_{\up}$ and as above there is a black arc $B_2$ with $t(B_2) = s(B_1)$ and $s(B_2) > s(B)$. We repeat this argument until $s(B_n) = s(B)+1$. 
\item We use the same argument as in a), but we have to stop with the repetition if $s(B_n)= 0$ or $s(B_n) = t(G)$ for $G$ a green arc. Note that arcs $A$ are green if and only if $s(A) < 0$.
\end{enumerate}
\end{bew}

\begin{defi}
Let $S$ be a standard tableau of type $(i_1, \dots, i_m)$. A flag $(F_{i_1} \subset \dots \subset F_{i_m})$ \textit{satisfies the conditions of} $\depG{S}$ if we have
\begin{enumerate}
	\item if the node labelled $F_i$ ($i>0$) is connected to a node labelled $F_j$ with $i<j$ via a black arc labelled $N^{-l}$, then $F_j = N^{-l} F_i$
	\item if the node labelled $F_i$ is the endpoint of a green arc labelled $e_l$, then $F_i = F_{i-1} + \left\langle e_l\right\rangle$
\end{enumerate}
\end{defi}

\begin{satz}[Graphical description of irreducible components] \label{abh(S)}~\\
Let $S$ be a standard tableau of type $(i_1, \dots, i_m)$.
Then the irreducible components $\comb{S}$ consist of all $N$-invariant flags satisfying the conditions of the dependence graph for $S$. 

The space $F_j$, $j>0$, is independent in $\comb{S}$ if and only if the node labelled $F_j$ is the node at the left end of a black connected component of the dependence graph for $S$, where a node without arcs is also a component.
\end{satz}

\begin{bew}
This follows from Theorem \ref{Komponenten}, first by induction on $b(B)$ for black arcs using Proposition \ref{bogen} \ref{schwarzer}) and after that by induction on $g(G)$ for green arcs using Proposition \ref{bogen} \ref{grüner}). For details see \cite{me}. 
\end{bew}

\begin{bem} \label{unabhängige}
The nodes in $\depG{S}$ at the left end of a black connected component of the dependence graph for $S$ coincide with the nodes that are at the left end of a black cup in the extended cup diagram. 

This holds, because $k$ is the left end of a connected component if and only if there is no arc $B$ in $\depG{S}$ such that $k=t(B)$. By the construction of the dependence graph, this is equivalent to $k \notin S_{\up}$, i.e. $k \in S_{\down}$ and $k$ is not a double entry, which means that $k$ is the left end of a cup.

Therefore, by Theorem \ref{abh(S)} the number of independents in $\comb{S}$ is the same as the number of black cups in $eC(S)$.
\end{bem}

\subsection{Dependence graphs describing generalised irreducible components}
\label{sec:DependenceGraphsForRowStrictTableaux}

\begin{defi}
Let $w$ be a row strict tableau of type $(i_1, \dots, i_m)$. The \textit{extended cup diagram} for $w$, $eC(w)$, is defined as follows: We add $n-k$ $\down$'s on the left and $k$ $\up$'s on the right of the weight sequence, i.e. ${\bf
a}=\underbrace{\down \dots \down}_{n-k}a_1 a_2 a_3\ldots a_n\underbrace{\up \dots \up}_{k}$. Then we connect $\down \up$ as usual until all of the nodes $1, \dots, n$ are connected. After that we delete the remaining ones of the added $\down$'s and $\up$'s. If an arc is starting at one of the newly added $\down$'s or ending at one of the newly added $\up$'s, we colour it green. 
\end{defi}

\begin{bsp}
$w = \young(6543,731)$, \quad $\bf{a} = \down \down \down \down \down \times \times \up \up \up \down \up \up \up$.

$$eC(w) = \begin{tabular}[c]{c} \begin{tikzpicture}
		\begin{scope}
		\draw (-0.4,0) node[above=-1.55pt] {$-1$};
		\draw (0,0) node[above=-1.55pt] {$0$};
		\draw (0.4,0) node[above=-1.55pt] {$1$};
		\draw (0.8, 0) node[above=-1.55pt] {$2$};
		\draw (1.2,0) node[above=-1.55pt] {$3$};
		\draw (1.6,0) node[above=-1.55pt] {$4$};
		\draw (2,0) node[above=-1.55pt] {$5$};
		\draw (2.4,0) node[above=-1.55pt] {$6$};
		\draw (2.8,0) node[above=-1.55pt] {$7$};
		\draw (3.2,0) node[above=-1.55pt] {$8$};
		\draw (0.8,0) node {$\times$};
		\draw (1.2,0) node {$\times$};
		\draw[dashed,thick,green!50!black] (2.8,0) arc (180:360:0.2);
		\draw (0.4,0) arc (180:360:0.6);
		\draw[dashed,thick,green!50!black] (0,0) arc (180:360:1.0);
		\draw[dashed,thick,green!50!black] (-0.4,0) arc (180:360:1.4);
		\end{scope}
		\end{tikzpicture} \end{tabular}$$

\end{bsp}

\begin{bsp}  If $n=2k$, the extended cup diagrams for a row strict tableau which is not a standard tableau do not coincide with the normal cup diagrams:\\ 
Let $n=4, k=2$. Then, in addition to the $2$ standard tableaux of type $(1,2,3,4)$, there are the following row strict tableaux of type $(1,2,3,4)$:\\
$w_1 = \young(21,43)$, $w_2 = \young(31,42)$, $w_3 = \young(32,41)$ and $w_4 = \young(41,32)$. \\
\begin{align*}
	eC(w_1) = \begin{tabular}[c]{c}\begin{tikzpicture}
			\begin{scope}
			\draw (-0.4,0) node[above=-1.55pt] {$-1$};
		  \draw (0,0) node[above=-1.55pt] {$0$};
			\draw (0.4,0) node[above=-1.55pt] {$1$};
			\draw (0.8, 0) node[above=-1.55pt] {$2$};
			\draw (1.2,0) node[above=-1.55pt] {$3$};
			\draw (1.6,0) node[above=-1.55pt] {$4$};
			\draw (2,0) node[above=-1.55pt] {$5$};
		  \draw (2.4,0) node[above=-1.55pt] {$6$};
			\draw[dashed,thick,green!50!black] (0,0) arc (180:360:0.2);
			\draw[dashed,thick,green!50!black] (-0.4,0) arc (180:360:0.6);
			\draw[dashed,thick,green!50!black] (1.6,0) arc (180:360:0.2);
			\draw[dashed,thick,green!50!black] (1.2,0) arc (180:360:0.6);
			\end{scope}
			\end{tikzpicture}	\end{tabular} \hspace{2cm}
	eC(w_2) = \begin{tabular}[c]{c}\begin{tikzpicture}
			\begin{scope}
			  \draw (0,0) node[above=-1.55pt] {$0$};
			\draw (0.4,0) node[above=-1.55pt] {$1$};
			\draw (0.8, 0) node[above=-1.55pt] {$2$};
			\draw (1.2,0) node[above=-1.55pt] {$3$};
			\draw (1.6,0) node[above=-1.55pt] {$4$};
			\draw (2,0) node[above=-1.55pt] {$5$};
		  \draw[dashed,thick,green!50!black] (0,0) arc (180:360:0.2);
			\draw[dashed,thick,green!50!black] (1.6,0) arc (180:360:0.2);
			\draw (0.8,0) arc (180:360:0.2);
			\end{scope}
			\end{tikzpicture}	\end{tabular}\\
				eC(w_3) = \begin{tabular}[c]{c} \begin{tikzpicture}
			\begin{scope}
			  \draw (0,0) node[above=-1.55pt] {$0$};
			\draw (0.4,0) node[above=-1.55pt] {$1$};
			\draw (0.8, 0) node[above=-1.55pt] {$2$};
			\draw (1.2,0) node[above=-1.55pt] {$3$};
			\draw (1.6,0) node[above=-1.55pt] {$4$};
			\draw (2,0) node[above=-1.55pt] {$5$};
		  \draw[dashed,thick,green!50!black] (0,0) arc (180:360:0.6);
			\draw[dashed,thick,green!50!black] (1.6,0) arc (180:360:0.2);
			\draw (0.4,0) arc (180:360:0.2);
			\end{scope}
			\end{tikzpicture}	\end{tabular} \hspace{2cm}
			eC(w_4) = \begin{tabular}[c]{c} \begin{tikzpicture}
			\begin{scope}
			  \draw (0,0) node[above=-1.55pt] {$0$};
			\draw (0.4,0) node[above=-1.55pt] {$1$};
			\draw (0.8, 0) node[above=-1.55pt] {$2$};
			\draw (1.2,0) node[above=-1.55pt] {$3$};
			\draw (1.6,0) node[above=-1.55pt] {$4$};
			\draw (2,0) node[above=-1.55pt] {$5$};
		  \draw[dashed,thick,green!50!black] (0,0) arc (180:360:0.2);
			\draw[dashed,thick,green!50!black] (0.8,0) arc (180:360:0.6);
			\draw (1.2,0) arc (180:360:0.2);
			\end{scope}
			\end{tikzpicture}	\end{tabular}
\end{align*}

\end{bsp}

\begin{defi} \label{defdepG}
Let $w$ be a row strict tableau of type $(i_1,\dots, i_m)$.
The \textit{dependence graph for} $w$, $\depG{w}$, is defined as follows: We have $m$ given nodes, numbered $i_1$ to $i_m$. 
To the left of these we add one node more than there are nodes in $eC(w)$ to the left of the node $1$ and number the new nodes by $\dots, -1, 0$. Analogously, to the right we add one node more than there are nodes in $eC(w)$ to the right of the node $n$ and number the new nodes by $n+1, n+2, \dots$. We label the nodes with $F_j$ for $j\in \{i_1,\dots, i_m\}$ and with $\{0\}$ for the node $0$; the remaining nodes are left unlabelled.

If $i_s$ is labelled with $\times$ in the extended cup diagram, then we connect $i_s -2$ and $i_s$ and label the resulting arc with $N^{-1}$.

Now, if $i$ and $j$ with $i<j \leq n$ are connected in the extended cup diagram, then in the dependence graph we connect the nodes $i-1$ and $j$ by an arc of the same colour. 
We label the black arcs with $N^{-l}$ and the green ones with $e_l$, where $l = \frac12(j-(i-1))$.

If $i$ and $j$ with $i\leq n < j$ are connected in the extended cup diagram, then we connect $i$ and $j+1$ with a green arc and label it with $f_l$, where $l={k+1-\frac12(j-i)}$.%$f_{k+1-b(B)}$.

We use the notation of Definition \ref{bogenbez} in this case as well.

\end{defi}

\begin{bsp}

$w = \young(6543,731)$\\
$\depG{w} =$\begin{tabular}[c]{c}
\begin{tikzpicture}
		\begin{scope}
	  \draw (-1,0) node[above=-1.55pt] {};
		\draw (0,0) node[above=-1.55pt] {};
		\draw (1,0) node[above=-1.55pt] {$\{0\}$};
		\draw (2, 0) node[above=-1.55pt] {$F_1$};
		\draw (3,0) node[above=-1.55pt] {$F_3$};
		\draw (4,0) node[above=-1.55pt] {$F_4$};
		\draw (5,0) node[above=-1.55pt] {$F_5$};
		\draw (6,0) node [above=-1.55pt] {$F_6$}; 
		\draw (7,0) node[above=-1.55pt] {$F_7$};
		\draw (2,0) arc (180:360:0.5);
	  \draw[dashed,thick,green!50!black] (7,0) arc (180:360:1.0);
		\draw (1,0) arc (180:360:1.5);
		\draw[dashed,thick,green!50!black] (0,0) arc (180:360:2.5);
		\draw[dashed,thick,green!50!black] (-1,0) arc (180:360:3.5);
		\draw (2.5,-0.7) node {$N^{-1}$};
		\draw (2.5,-1.7) node {$N^{-2}$};
		\draw (2.5,-2.7) node {$e_3$};
		\draw (2.5,-3.7) node {$e_4$};
			\draw (8.0,-1.2) node {$f_3$};
		\end{scope}
		\end{tikzpicture}
		\end{tabular}
\end{bsp}

\begin{defi}
Let $w$ be a row strict tableau of type $(i_1, \dots, i_m)$. A flag $(F_{i_1} \subset \dots \subset F_{i_m})$ \textit{satisfies the conditions of} $\depG{w}$ if we have
\begin{enumerate}
	\item if the node labelled $F_i$ ($i>0$) is connected to a node labelled $F_j$ with $i<j$ via a black arc labelled $N^{-l}$, then $F_j = N^{-l} F_i$
	\item if the node labelled $F_i$ is the endpoint of a green arc labelled $e_l$, then $F_i = F_{i-1} + \left\langle e_l\right\rangle$
	\item if the node labelled $F_i$ is the starting point of a green arc labelled $f_l$, then $F_i = F_{i-1} + \left\langle f_l\right\rangle$
\end{enumerate}
\end{defi}

The next theorem connects, similarly to the one before, the dependence graphs with the generalised irreducible components. 

\begin{satz}[Graphical description of generalised irreducible components] \label{abh(w)} ~\\
Let $w$ be a row strict tableau of type $(i_1, \dots, i_m)$.

Then $\stb{w}$ consist of all $N$-invariant flags satisfying the conditions of the dependence graph for $w$.

The space $F_j$, $j>0$, is independent in $\stb{w}$ if and only if the node labelled $F_j$ is the node at the left end in a black connected component of the dependence graph for $w$.
\end{satz}

\begin{bew}
Let $W_1 = (w_{\up} \cap S(w)_{\down})\setminus w_{\times}$, $W_2 = (w_{\down} \cap S(w)_{\up}) \setminus w_{\times}$. 

By Theorem \ref{row strict} and Theorem \ref{abh(S)} we only have to show, that for $N$-invariant flags the conditions of $\depG{w}$ are equivalent to the conditions of $\depG{S(w)}$ together with the additional condition $F_i = \mathcal{F}_i(w)$ for $i \in W_1$. 
By comparison of the constructions we get that $\depG{w}$ and $\depG{S(w)}$ agree except for black arcs $B$ in $\depG{S(w)}$ with $s(B)+1 \in W_1$ and $t(B) \in W_2$ and green arcs $G$ in $\depG{w}$ with $t(G) \in W_1$ or $s(G) \in W_2$, respectively.

Inductively by using Proposition \ref{bogen} \ref{grüner}) one can see that for green arcs in $\depG{w}$ with $t(G) \in W_1$ the conditions of $\depG{w}$ agree with those of $\depG{S(w)}$ together with the additional condition. Then one can inductively show the same for green arcs in $\depG{w}$ with $s(G) \in W_2$ by using Proposition \ref{bogen} \ref{schwarzer}) and the situation the arcs are in:

\begin{center}\begin{tikzpicture}
		\begin{scope}
	  \draw (0,0) node[above=-1.55pt] {};
		\draw (2, 0) node[above=-1.55pt] {$s(B)~$};
		\draw (3,0) node[above=-1.55pt] {~~~$s(B)+1$};
		\draw (6,0) node [above=-1.55pt] {$t(B)$}; 
		\draw (9,0) node[above=-1.55pt] {};
		\draw (2,0) arc (180:360:2);
	  \draw[dashed,thick,green!50!black] (6,0) arc (180:360:1.5);
		\draw[dashed,thick,green!50!black] (0,0) arc (180:360:1.5);
		\end{scope}
\end{tikzpicture} \end{center}

Here the black arc $B$ is the one in $\depG{S(w)}$ with $s(B)+1 \in W_1$ and $t(B) \in W_2$. The green arcs are in $\depG{w}$. For details see \cite{me}.
\end{bew}

\begin{bem} \label{unabhängige2}
The nodes at the left end of a black connected component of the dependence graph for $w$ coincide with the left ends of black arcs in $eC(w)$. This follows in the same way as in Remark \ref{unabhängige}.
Therefore, by Theorem \ref{abh(w)} the number of independents in $\stb{w}$ is the same as the number of black cups in $eC(w)$.
\end{bem}

\subsection{Dependence graphs for intersections}
\label{sec:DependenceGraphsForIntersections}

\begin{defi}
Let $w$ and $w'$ be row-strict tableaux of type $(i_1, \dots, i_m)$.
The \textit{dependence graph for} $(w,w')$, $\depG{w,w'}$, is constructed by reflecting the dependence graph for  $w'$ across the horizontal axis and putting it on top of the dependence graph for $w$.
\end{defi}

\begin{bsp}
Let again $w = \young(6543,731)$ and let $w'=\young(7543,631)$ (which is the $S$ from Example \ref{bspS}).\\ 
$\depG{w,w'} = $		\begin{tabular}[c]{c}
		\begin{tikzpicture}
		\begin{scope}
	  \draw (-1,0) node[above=-1.55pt] {};
		\draw (0,0) node[above=-1.55pt] {};
		\draw (1,0) node[above=-1.55pt] {$\{0\}$};
		\draw (2, 0) node[above=-1.55pt] {$F_1$};
		\draw (3,0) node[above=-1.55pt] {$F_3$};
		\draw (4,0) node[above=-1.55pt] {$F_4$};
		\draw (5,0) node[above=-1.55pt] {$F_5$};
		\draw (6,0) node [above=-1.55pt] {$F_6$}; 
		\draw (7,0) node[above=-1.55pt] {$F_7$};
		\draw (2,0) arc (180:360:0.5);
	  \draw[dashed,thick,green!50!black] (7,0) arc (180:360:1.0);
		\draw (1,0) arc (180:360:1.5);
		\draw[dashed,thick,green!50!black] (0,0) arc (180:360:2.5);
		\draw[dashed,thick,green!50!black] (-1,0) arc (180:360:3.5);
		\draw (2.5,-0.7) node {$N^{-1}$};
		\draw (2.5,-1.7) node {$N^{-2}$};
		\draw (2.5,-2.7) node {$e_3$};
		\draw (2.5,-3.7) node {$e_4$};
			\draw (8.0,-1.2) node {$f_3$};
			%andere hälfte:
					\draw (2,0) arc (-180:-360:0.5);
		\draw (1,0) arc (-180:-360:1.5);
		\draw (5,0) arc (-180:-360:1);
		\draw[dashed,thick,green!50!black] (0,0) arc (-180:-360:2.5);
		\draw (2.5,0.7) node {$N^{-1}$};
		\draw (6,1.2) node {$N^{-1}$};
		\draw (2.5,1.7) node {$N^{-2}$};
		\draw (2.5,2.7) node {$e_3$};
  	\end{scope}
		\end{tikzpicture} \end{tabular}
\end{bsp}

\begin{defi}
Let $w,w'$ be row strict tableaux of type $(i_1, \dots, i_m)$. A flag $(F_{i_1} \subset \dots \subset F_{i_m})$ \textit{satisfies the conditions of} $\depG{w,w'}$ if we have
\begin{enumerate}
	\item if the node labelled $F_i$ ($i>0$) is connected to a node labelled $F_j$ with $i<j$ via a black arc labelled $N^{-l}$, then $F_j = N^{-l} F_i$
	\item if the node labelled $F_i$ is the endpoint of a green arc labelled $e_l$, then $F_i = F_{i-1} + \left\langle e_l\right\rangle$
	\item if the node labelled $F_i$ is the starting point of a green arc labelled $f_l$, then $F_i = F_{i-1} + \left\langle f_l\right\rangle$
\end{enumerate}
\end{defi}

\begin{kor}\label{Abh(w,w')}
Let $w$ and $w'$ be row-strict tableaux of type $(i_1, \dots, i_m)$.

Then $\stb{w} \cap \stb{w'}$ consists of all $N$-invariant flags satisfying the conditions of the dependence graph for $(w,w')$. 
The space $F_j$, $j>0$, is independent in $\comb{w} \cap \comb{w'}$ if and only if the node labelled $F_j$ is the node at the left end in a connected component of the dependence graph for $(w,w')$.
\end{kor}
\begin{bew}
$F_{\bullet}$ is in $\comb{w} \cap \comb{w'}$ if and only if the conditions of $\comb{w}$ and the conditions of  $\comb{w'}$ are satisfied. But these conditions are given by the associated dependence graph. If the dependence graphs are put on top of each other, then both conditions are satisfied simultaneously.
\end{bew}

\section{Circle diagrams}
\label{sec:CupCapDiagrams}
Our next goals are Theorem \ref{B} and Theorem \ref{C}. While dependence graphs describe the structure of generalised irreducible components, we need extended cup diagrams and circle diagrams to describe the cohomology of generalised irreducible components or their intersections, respectively. 

Following \cite[5.4]{stroppel2} we construct circle diagrams out of cup diagrams and colour them:
\begin{defi}
Let $w, w'$ be row strict tableaux. We define $CC(w,w')$, the \textit{circle diagram for} $(w,w')$, as follows: We reflect $eC(w')$ and put it on top of $eC(w)$.
If there are more points in $eC(w)$ than in $eC(w')$ or vice versa, we connect in $eC(w')$ the ones on the right with the ones on the left via an green arc in the only possible crossingless way. The construction up to here is called $eC(w,w')$. 
If there is at least one green arc in a connected component, we colour the whole component green.
If there is more than one left outer point, i.e. a point $p$ with $p<1$, or more than one right outer point, i.e. a point $p$ with $p>n$, in a connected component, we colour the whole component red. (Red lines are represented by thick lines for the black and white version.)
From now on we also use the term \textit{circle} for connected component. 
\end{defi}

\begin{bsp}
$w = \young(6543,731)$,\; $w'=\young(7543,631)$,\; $w''=\young(7653,431)$ 

$$CC(w,w') =  \begin{tabular}[c]{c} \begin{tikzpicture}
		\begin{scope}
		\draw (-0.4,0) node[above=-1.55pt] {$-1$};
		\draw (0,0) node[above=-1.55pt] {$0$};
		\draw (0.4,0) node[above=-1.55pt] {$1$};
		\draw (0.8, 0) node[above=-1.55pt] {$2$};
		\draw (1.2,0) node[above=-1.55pt] {$3$};
		\draw (1.6,0) node[above=-1.55pt] {$4$};
		\draw (2,0) node[above=-1.55pt] {$5$};
		\draw (2.4,0) node[above=-1.55pt] {$6$};
		\draw (2.8,0) node[above=-1.55pt] {$7$};
		\draw (3.2,0) node[above=-1.55pt] {$8$};
		\draw (0.8,0) node {$\times$};
		\draw (1.2,0) node {$\times$};
		\draw[dashed,thick,green!50!black] (2.8,0) arc (180:360:0.2);
		\draw (0.4,0) arc (180:360:0.6);
		\draw[dashed,thick,green!50!black] (0,0) arc (180:360:1.0);
		\draw[dashed,thick,green!50!black] (-0.4,0) arc (180:360:1.4);
		\draw[dashed,thick,green!50!black] (2.4,0) arc (-180:-360:0.2);
		\draw (0.4,0) arc (-180:-360:0.6);
		\draw[dashed,thick,green!50!black] (0,0) arc (-180:-360:1);
		\draw[dashed,thick,green!50!black] (-0.4,0) arc (-180:-360:1.8);
		\end{scope}
		\end{tikzpicture}\end{tabular}$$
		
$$CC(w,w'') =  \begin{tabular}[c]{c}  \begin{tikzpicture}
		\begin{scope}
			\draw[very thick,red] (2.8,0) arc (180:360:0.2);
		\draw [very thick,red](0.4,0) arc (180:360:0.6);
		\draw[very thick,red] (0,0) arc (180:360:1.0);
		\draw[very thick,red] (-0.4,0) arc (180:360:1.4);
		\draw[very thick,red] (1.6,0) arc (-180:-360:0.2);
		\draw[very thick,red] (0.4,0) arc (-180:-360:1.0);
		\draw[very thick,red] (0,0) arc (-180:-360:1.4);
		\draw[very thick,red] (-0.4,0) arc (-180:-360:1.8);
		\draw (-0.4,0) node[above=-1.55pt] {$-1$};
		\draw (0,0) node[above=-1.55pt] {$0$};
		\draw (0.4,0) node[above=-1.55pt] {$1$};
		\draw (0.8, 0) node[above=-1.55pt] {$2$};
		\draw (1.2,0) node[above=-1.55pt] {$3$};
		\draw (1.6,0) node[above=-1.55pt] {$4$};
		\draw (2,0) node[above=-1.55pt] {$5$};
		\draw (2.4,0) node[above=-1.55pt] {$6$};
		\draw (2.8,0) node[above=-1.55pt] {$7$};
		\draw (3.2,0) node[above=-1.55pt] {$8$};
		\draw (0.8,0) node {$\times$};
		\draw (1.2,0) node {$\times$};	
		\end{scope}
		\end{tikzpicture}  \end{tabular}$$
\end{bsp}

\begin{bem}\label{unabhängige3}
The nodes at the left end in a black connected component of the dependence graph for $(w,w')$ coincide with the left points of the black circles in $CC(w,w')$. This follows from the same argument as in Remark \ref{unabhängige} and Remark \ref{unabhängige2}.

Therefore, by Corollary \ref{Abh(w,w')} the number of independents in $\stb w \cap \stb{w'}$ is the same as the number of black circles in $CC(w,w')$.
\end{bem}

In the following section the theorems often have the assumption $\stb w \cap \stb{w'} \neq \emptyset$. 
The first question when considering the pairwise intersection of generalised irreducible is whether the intersection is empty.
The following theorem gives an equivalent condition for this in terms of circle diagrams.

\begin{satz} \label{rote}
Let $w, w'$ be row strict tableaux of type $(i_1, \dots, i_m)$. Then we have that $\stb w \cap \stb{w'} = \emptyset$ if and only if there is at least one red circle in $CC(w,w')$.
\end{satz}

\begin{bew}
In the following, we call a green arc $G$ a left green arc if $t(G) \in \{1, \dots, n\}$ and a right green arc if $s(G) \in \{1, \dots, n\}$. We denote 
	\begin{align*}r(G) &= \begin{cases} 
	t(G), &\text{if } G \text{ is a left green arc}\\
	s(G), &\text{if } G \text{ is a right green arc} 
	\end{cases}.
\end{align*}
An extended arc sequence in $eC(w,w')$ or $\depG{w,w'}$ is a sequence of arcs $B_1, \dots, B_l$ in $\depG{w,w'}$ such that for $1 \leq i \leq l-1$ exactly one of the following conditions hold:
	\[t(B_i) = s(B_{i+1}), \quad t(B_i) = t(B_{i+1}), \quad s(B_i) = s(B_{i+1}),\quad s(B_i) = t(B_{i+1}).
\] 

Assume $CC(w,w')$ contains a red circle. Consider this circle in $eC(w,w')$, i.e. when the colours are the colours from $eC(w)$ and $eC(w')$.

By considering the different connection possibilities one can show that there are two green arcs $H_1$, $H_2$ in this circle in $eC(w,w')$ such that $H_1$ and $H_2$ are both left green arcs or both right green arcs, they are both above the x-axis or both below and they are connected by an extended black arc sequence.

This results in a black extended arc sequence from $r(G_1)-1$ to $r(G_2)$ and one from $r(G_1)$ to $r(G_2)-1$, where $G_1$ and $G_2$ are the images of $H_1$ and $H_2$, respectively, in $\depG{w,w'}$. But then the two conditions that $\depG{w,w'}$ imposes on $F_{r(G_2)}$ contradict each other. Thus, $\stb w \cap \stb{w'} = \emptyset$.

Conversely, we assume that $CC(w,w')$ contains only black and green circles. Analogously to \cite[Lemma 19]{stroppel} we construct a row strict tableau $w''$, such that $\fx{w''} \in \stb w \cap \stb {w'}$:  %(vgl Lemma \ref{fixed points}

In $CC(w,w')$ we mark all points with $j \leq 0$ by $\down$ and all points with $j >n$ by $\up$. Now we assign a $\up$ or a $\down$ to all $j \in \{1, \dots, n\} \setminus w_{\times}$ as follows:

If we have a black circle, we mark one point arbitrarily and then we follow the circle and alternatingly mark the points where we meet the x-axis by $\up$ or $\down$, such that each arc has a $\up$ at one end and a $\down$ at the other.
For green circles we start with a point already marked and go on as for black circles.

Now we define a row strict tableau $w''$ by the fact that $j \in w''_{\up}$ if and only if $j$ is a double entry or marked by $\up$ and $j \in w''_{\down}$ if and only if $j$ is a double entry or marked by $\down$. 

In Lemma \ref{fixed points} we already showed that $\fx{w''}$ is $N$-invariant.
Furthermore, one can calculate that $\fx{w''}$ satisfies the conditions of $\depG{w,w'}$.
\end{bew}

\begin{bem}
In the proof one can see a fact similar to \cite[Lemma 19]{stroppel}: If there are no red circles in $CC(w,w')$, the number of fixed points contained in $\stb w \cap \stb {w'}$ is at least $2^x$ for $x$ the number of black circles in $CC(w,w')$.
Indeed, there are two possible choices for every black circle and only one for every green circle. 
\end{bem}

\section{Iterated fibre bundles} 
\label{sec:iteratedFib}

By \cite{gunn} to each complex projective variety $X$ we can associated a topological Hausdorff 
space $X^{an}$, the associated analytic space with the same underlying set. If the projective variety is smooth, $X^{an}$ is a complex manifold. 

Following \cite{fung} we consider iterated fibre bundles:
\begin{defi} 
A space $X_1$ is an \emph{iterated fibre bundle} of
base type $(B_1, \dots, B_l)$ if there exist spaces $X_1, B_1, X_2,
B_2, \dots, X_l, B_l, X_{l+1}=pt$ and maps $p_1, p_2,
\dots, p_{l}$ such that $p_j:X_j \to B_j$ is a fibre bundle with
typical fibre $X_{j+1}=F_j$. Here fibre bundle means topological fibre bundle in the sense of \cite{hatcher}.
\end{defi}

The following theorem is the main theorem of this section. It generalises \cite{stroppel} and \cite{fung}.

\begin{satz} \label{bundles}
Let $w, w'$ be row strict tableaux of type $(i_1, \dots, i_m)$ and assume $\stb w \cap \stb {w'} \neq \emptyset$.
Then $(\stb w)^{an}$ and $(\stb w \cap \stb {w'})^{an}$ are iterated bundles of base type $(\mathbb{CP}^1, \dots, \mathbb{CP}^1)$, where there are as many terms as there are independent nodes in the associated dependence graph.
\end{satz}

For the rest of the section we drop the $^{an}$ from the notation and always work with the analytic spaces.

For the proof of Theorem \ref{bundles} we first take a look at two important special cases:
\begin{bsp} \label{special cases}
Let $n=4$, $k=2$. As mentioned in Example \ref{gewünschtes}, in this case there are two standard tableaux of type $(1,2,3,4)$: $S_1 = \young(43,21)$ and $S_2 = \young(42,31)$. \\ 
For these we have the following associated extended cup diagrams:
\begin{align*}
	eC(S_1) = \begin{tabular}[c]{c}\begin{tikzpicture}
			\begin{scope}
			\draw (0.4,0) node[above=-1.55pt] {$1$};
			\draw (0.8, 0) node[above=-1.55pt] {$2$};
			\draw (1.2,0) node[above=-1.55pt] {$3$};
			\draw (1.6,0) node[above=-1.55pt] {$4$};
			\draw (0.8,0) arc (180:360:0.2);
			\draw (0.4,0) arc (180:360:0.6);
			\end{scope}
			\end{tikzpicture}	\end{tabular} \hspace{2cm}
	eC(S_2) = \begin{tabular}[c]{c}\begin{tikzpicture}
			\begin{scope}
			\draw (0.4,0) node[above=-1.55pt] {$1$};
			\draw (0.8, 0) node[above=-1.55pt] {$2$};
			\draw (1.2,0) node[above=-1.55pt] {$3$};
			\draw (1.6,0) node[above=-1.55pt] {$4$};
			\draw (0.4,0) arc (180:360:0.2);
			\draw (1.2,0) arc (180:360:0.2);
			\end{scope}
			\end{tikzpicture} \end{tabular}
\end{align*}

\begin{enumerate}[1)]
	\item \label{S_2} Firstly, we consider $\com{S_2}$:
As one can see for example from the associated dependence graph, we have
\begin{align*}
	\com{S_2} = \{ F_1 \subset N^{-1}(\{0\}) \subset F_3 \subset N^{-2}(\{0\}) \}.
\end{align*}
We define $p: \com{S_2} \to \mathbb{P}(N^{-1}(\{0\})) = \CP^1$ via
	\[\big(F_1 \subset N^{-1}(\{0\}) \subset F_3 \subset N^{-2}(\{0\})\big) \mapsto F_1.
\]
This is a trivial fibre bundle with fibre a component of a smaller Springer fibre, namely the one with associated cup diagram  
\begin{tikzpicture}
			\begin{scope}
			\draw (0.4,0) node[above=-1.55pt] {$1$};
			\draw (0.8, 0) node[above=-1.55pt] {$2$};
			\draw (0.4,0) arc (180:360:0.2);
			\end{scope}
			\end{tikzpicture}: \\
Let $N_1:= N |_{\left\langle e_1, f_1\right\rangle}$ and let $N_2$ be the map induced by $N$ on $\C^4/\ker N \to \C^4/\ker N$. We take as trivialising neighbourhood the whole base space and get the following trivialisation:
\begin{align*}
	\{F_1 \subset N_1^{-1}(\{0\})\} \times \{G_1 \subset N_2^{-1}(\{0\})\} &\to p^{-1}\big(\{F_1 \subset N^{-1}(\{0\})\}\big) \\
	\big(F_1 \subset N^{-1}(\{0\}), G_1 \subset N_2^{-1}(\{0\}) \big) &\mapsto \big(F_1 \subset N_1^{-1}(\{0\}) \subset N_1^{-1}(\{0\}) + G_1 \\ &\subset N_1^{-1}(\{0\}) + N_2^{-1}(\{0\}) = \C^2 \oplus \C^2 = \C^4\big) \\
	\big(F_1 \subset N_1^{-1}(\{0\}), F_3/\ker N \subset \C^4/\ker N\big) &\mapsfrom \big(F_1 \subset N^{-1}(\{0\}) \subset F_3 \subset \C^4\big)
\end{align*} This is a homeomorphism and commutes with the projections.

\item \label{S_1} Now we consider the space $\com{S_1} = \{F_1 \subset F_2 \subset N^{-1}(F_1) \subset N^{-2}(\{0\})\}$. 
Again we define $p: \com{S_1} \to \mathbb{P}\big(N^{-1}(\{0\})\big) = \CP^1$ via		
	\[\big(F_1 \subset F_2 \subset N^{-1}(F_1) \subset N^{-2}(\{0\})\big) \mapsto F_1.
\]	We choose the standard covering of $\CP^1$: 
\begin{align*}
	U_1 := \{(x:y) | x \neq 0 \} \text{ and } U_2:= \{(x:y) | y \neq 0 \}
\end{align*}
We consider $(1: \lambda) \in U_1$. With our identifications this corresponds to $\left\langle e_1 + \lambda f_1\right\rangle \subset \left\langle e_1, f_1\right\rangle$. We have $N^{-1}\left\langle e_1 + \lambda f_1 \right\rangle = \left\langle e_1, f_1, e_2 + \lambda f_2 \right\rangle = \left\langle e_1 + \lambda f_1, f_1, e_2 + \lambda f_2\right\rangle$ and therefore we obtain $N^{-1}\left\langle e_1 + \lambda f_1 \right\rangle/\left\langle e_1 + \lambda f_1 \right\rangle = \left\langle f_1, e_2+ \lambda f_2\right\rangle$. We denote the isomorphism $\C^2 \to \left\langle f_1, e_2+\lambda f_2\right\rangle$ given by mapping the standard basis to $f_1, e_2+ \lambda f_2$ by $\alpha_{\lambda}$. 
We get the following trivialisation for $U_1$: 
\begin{align*}
	U_1 \times \{L \subset \C^2\} &\to p^{-1}(U_1) \\
	\big( (1: \lambda), L \subset \C^2\big) &\mapsto \big(\left\langle e_1+ \lambda f_1 \right\rangle \subset \left\langle e_1 + \lambda f_1 \right\rangle + \alpha_{\lambda}(L) \subset N^{-1}(\left\langle e_1  + \lambda f_1 \right\rangle) \subset \C^4\big)
\end{align*}
This map commutes with the projections, and it is a homeomorphism with inverse
\begin{align*}
	p^{-1}(U_1) &\to  U_1 \times \{L \subset \C^2\} \\
	\Big(&\left\langle  e_1+ \lambda f_1 \right\rangle \subset F_2 \subset N^{-1}(\left\langle e_1+ \lambda f_1 \right\rangle) \subset \C^4\Big)  
	\\ &\mapsto \big(F_1 \subset N^{-1}(\{0\})\big)\\ &\quad\quad \times  \Big(\alpha_{\lambda}^{-1}(F_2/\left\langle e_1+ \lambda f_1 \right\rangle) \subset \alpha_{\lambda}^{-1}(N^{-1}(\left\langle e_1+ \lambda f_1 \right\rangle)/\left\langle e_1+ \lambda f_1 \right\rangle) = \C^2\Big). 
\end{align*}
For $U_2$ we get an analogous trivialisation.
Altogether, $\com{S_1}$ is a non-trivial fibre bundle with fibre a component of a smaller Springer fibre, again the one with cup diagram \begin{tikzpicture}
			\begin{scope}
			\draw (0.4,0) node[above=-1.55pt] {$1$};
			\draw (0.8, 0) node[above=-1.55pt] {$2$};
			\draw (0.4,0) arc (180:360:0.2);
			\end{scope}
			\end{tikzpicture}.
\end{enumerate}
\end{bsp}

\begin{defi} We define the \textit{cup diagram decomposition} as follows:\\
If $C$ is a cup diagram with all cups nested inside a single cup $C'$, we say $C$ is in nested position. If $C$ is in nested position, then the cup diagram decomposition is $C = C' \ast D$ where $D$ is $C$ with $C'$ removed and numbering adjusted. \\
If this is not the case, then the cup diagram decomposition is $C = D \ast D'$, where $D$ consists of the cup $B$ with $s(B) = 1$ and of all the cups nested inside $B$ and $D'$ consists of the remaining cups with the numbering adjusted.

For an extended cup diagram $C$ let $|C|$ be twice the number of cups.
\end{defi}

\begin{bsp}
\begin{align*}
	\begin{tabular}[c]{c} \begin{tikzpicture}
			\begin{scope}
			\draw (-0.4,0) node[above=-1.55pt] {$1$};
			\draw (0,0) node[above=-1.55pt] {$2$};
			\draw (0.4,0) node[above=-1.55pt] {$3$};
			\draw (0.8, 0) node[above=-1.55pt] {$4$};
			\draw (1.2,0) node[above=-1.55pt] {$5$};
			\draw (1.6,0) node[above=-1.55pt] {$6$};
			\draw (2,0) node[above=-1.55pt] {$7$};
			\draw (2.4,0) node[above=-1.55pt] {$8$};
			\draw (0.4,0) arc (180:360:0.2);
			\draw (1.2,0) arc (180:360:0.2);
			\draw (0,0)   arc (180:360:1.0);
			\draw (-0.4,0) arc (180:360:1.4);
			\end{scope}
			\end{tikzpicture} \end{tabular}
			&=
			\begin{tabular}[c]{c}	\begin{tikzpicture}
			\begin{scope}
			\draw (0.4,0) node[above=-1.55pt] {$1$};
			\draw (0.8, 0) node[above=-1.55pt] {$2$};
			\draw (0.4,0) arc (180:360:0.2);
			\end{scope}
			\end{tikzpicture} \end{tabular}
			\ast
			\begin{tabular}[c]{c}	\begin{tikzpicture}
			\begin{scope}
			\draw (0,0) node[above=-1.55pt] {$1$};
			\draw (0.4,0) node[above=-1.55pt] {$2$};
			\draw (0.8, 0) node[above=-1.55pt] {$3$};
			\draw (1.2,0) node[above=-1.55pt] {$4$};
			\draw (1.6,0) node[above=-1.55pt] {$5$};
			\draw (2,0) node[above=-1.55pt] {$6$};
			\draw (0.4,0) arc (180:360:0.2);
			\draw (1.2,0) arc (180:360:0.2);
			\draw (0,0)   arc (180:360:1.0);
			\end{scope}
			\end{tikzpicture}	\end{tabular}		
\end{align*}

\begin{align*}\begin{tabular}[c]{c}
			\begin{tikzpicture}
			\begin{scope}
			\draw (-2.4,0) node[above=-1.55pt] {$1$};
			\draw (-2,0) node[above=-1.55pt] {$2$};
			\draw (-1.6,0) node[above=-1.55pt] {$3$};
			\draw (-1.2,0) node[above=-1.55pt] {$4$};
			\draw (-0.8,0) node[above=-1.55pt] {$5$};
			\draw (-0.4,0) node[above=-1.55pt] {$6$};
			\draw (0,0) node[above=-1.55pt] {$7$};
			\draw (0.4,0) node[above=-1.55pt] {$8$};
			\draw (0.8, 0) node[above=-1.55pt] {$9$};
			\draw (1.2,0) node[above=-1.55pt] {$10$};
			\draw (1.6,0) node[above=-1.55pt] {$11$};
			\draw (2,0) node[above=-1.55pt] {$12$};
			\draw (0.4,0) arc (180:360:0.2);
			\draw (1.2,0) arc (180:360:0.2);
			\draw (0,0)   arc (180:360:1.0);
			\draw (-0.8,0) arc (180:360:0.2);
			\draw (-2,0) arc (180:360:0.2);
			\draw (-2.4,0) arc (180:360:0.6);
			\end{scope}
			\end{tikzpicture}		\end{tabular}	
			&=
			\begin{tabular}[c]{c}  \begin{tikzpicture}
			\begin{scope}
			\draw (-2.4,0) node[above=-1.55pt] {$1$};
			\draw (-2,0) node[above=-1.55pt] {$2$};
			\draw (-1.6,0) node[above=-1.55pt] {$3$};
			\draw (-1.2,0) node[above=-1.55pt] {$4$};
			\draw (-2,0) arc (180:360:0.2);
			\draw (-2.4,0) arc (180:360:0.6);
			\end{scope}
			\end{tikzpicture}	\end{tabular}
			\ast
			\begin{tabular}[c]{c}			\begin{tikzpicture}
			\begin{scope}
			\draw (-0.8,0) node[above=-1.55pt] {$1$};
			\draw (-0.4,0) node[above=-1.55pt] {$2$};
			\draw (0,0) node[above=-1.55pt] {$3$};
			\draw (0.4,0) node[above=-1.55pt] {$4$};
			\draw (0.8, 0) node[above=-1.55pt] {$5$};
			\draw (1.2,0) node[above=-1.55pt] {$6$};
			\draw (1.6,0) node[above=-1.55pt] {$7$};
			\draw (2,0) node[above=-1.55pt] {$8$};
			\draw (0.4,0) arc (180:360:0.2);
			\draw (1.2,0) arc (180:360:0.2);
			\draw (0,0)   arc (180:360:1.0);
			\draw (-0.8,0) arc (180:360:0.2);
			\end{scope}
			\end{tikzpicture}	\end{tabular}
\end{align*}
 
\end{bsp}

\begin{defi}
For $n$ even, by $Y^n_S$ we denote the irreducible component of the Springer fibre associated to a standard tableau $S$ of type $(1, \dots, n)$ and of shape $(\frac{n}2, \frac{n}2)$.
\end{defi}
Note that we can associate an unique standard tableau $S$ of type $(1, \dots, n)$ and of shape $(\frac{n}2, \frac{n}2)$ to each cup diagram $C$ consisting of black cups such that $eC(S) = C$ by writing the numbers at the right end of the cups in the bottom line and the other ones in the top line.

\begin{lemma} \label{case1} Let $n$ be even and $Y^n_S$ be a component of the Springer fibre such that $eC(S)$ is not in nested position and let $eC(S) = eC(R) \ast eC(T)$ its cup diagram decomposition. Let $r = |eC(R)|$ and $t = |eC(T)|$. Then $Y^n_S$ is the total space of a trivial fibre bundle with base space $Y^r_R$ and fibre $Y^t_T$.
\end{lemma}
\begin{bew} We have
\begin{align*}
	Y^n_S = \{(F_1 \subset \dots \subset F_r \subset F_{r+1} \subset \dots \subset F_n) = F_{\bullet} | F_{\bullet} \text{ is } N \text{-invariant} \\ \text{and satisfy the conditions of } \depG{S}\}
\end{align*}
From the conditions of $\depG{S}$ we know that $F_r = N^{-\frac r2} \{0\}$, thus it is a fixed subspace of $\C^n$. Now, we consider the map
\begin{eqnarray*}
	p: Y^n_S &\to& Y^r_R \\
	(F_1 \subset \dots \subset F_n) &\mapsto& \left(F_1 \subset \dots \subset F_r = N^{- \frac r2}\{0\}\right).
\end{eqnarray*}
Analogous to Example \ref{special cases} \ref{S_2}) this is a trivial fibre bundle with fibre  $Y^t_T$.
\end{bew}

\begin{lemma}\label{case2}
Let $n$ be even and $Y^n_S$ be a component of the Springer fibre such that $eC(S)$ is in nested position and let $eC(S) = C \ast eC(T)$ its cup diagram decomposition. Let $t = |eC(T)|$. Then $Y^n_S$ is the total space of a non-trivial fibre bundle with base space $\CP^1$ and fibre $Y^t_T$.
\end{lemma}
\begin{bew}
Again, we consider the map
\begin{align*}
	p: Y^n_S \to \CP^1 \\
	(F_1 \subset \dots \subset F_n) \mapsto \left(F_1 \subset N^{-1}\{0\}\right).
\end{align*}
As in Example \ref{special cases} \ref{S_1}) this is a non-trivial fibre bundle with fibre  $Y^t_T$.
\end{bew}

\begin{bem}
The results stated in Lemmas \ref{case1} and \ref{case2} are shown in wider generality in \cite{fresse+melnikov} and in \cite{fresse}. In the latter reference, they are combined to define a family of smooth components of Springer fibres which contains the two-block components as a subfamily. 
\end{bem}

\begin{defi}
We call an iterated fibre bundle of base type $(\underbrace{\mathbb{CP}^1, \dots, \mathbb{CP}^1}_{l})$ an \textit{$l$-bundle}.
\end{defi}

\begin{lemma} \label{trivialfib}
Let $E = B \times F$ be a trivial fibre bundle and let $B$ be an $l_1$-bundle and let $F$ be an $l_2$-bundle. Then $E$ is an $(l_1+l_2)$-bundle.
\end{lemma}
\begin{bew}
Let $B=B_1, \dots, B_{l_1}, B_{l_1+1} = \pt$ be the total spaces of the iterated fibre bundles for $B$ and let $F = F_1, \dots, F_{l_2}, F_{l_2+1} = \pt$ the ones for $F$. Then $E$ is an iterated fibre bundle with total spaces $E, B_2 \times F, \dots, B_{l_1} \times F, \pt \times F = F, F_2, \dots, F_{l_2}, \pt$ and the associated maps.
\end{bew}

%%%%%%%%%%%%%%%%%%%%%%%%%%%%%%%%%%
Since the cup diagram decomposition just works for diagrams with only black cups, we need the following proposition to reduce to this case. It connects a generalised irreducible component to an irreducible component associated to a cup diagram that arises from the one of the generalised component by deleting all green cups.

\begin{prop} \label{löschegrüne} Let $w$ be row strict tableau of type $(1, \dots, n)$ of shape $(n-k,k)$ and let $n'$ be twice the number of black cups in $eC(w)$. Then there is a standard tableau $S$ of type $(1, \dots, n')$ and shape $(\frac{n'}2,\frac{n'}2)$ such that we have an isomorphism of varieties $\st w \to \com S$.  
\end{prop}
\begin{bew}
Let $S$ be the standard tableau corresponding to the cup diagram that arises from $eC(w)$ by deleting the green cups.
In $\depG{w}$ let 
\begin{align*}
	J_1 &:= \{j \in \{1, \dots, n\} | j = t(G) \text{ for } G \text{ a green arc} \} \text{ and} \\J_2 &:= \{j \in \{1, \dots, n\} | j = s(G) \text{ for } G \text{ a green arc} \}.
\end{align*}
 Let $J := J_1 \cup J_2$. We define $V_1 := \left\langle e_{b(G)} | t(G) = j \text{ for some } j \in J_1 \right\rangle$ and \\
 $V_2 := \left\langle f_{k+1-b(G)} | s(G) = j \text{ for some } j \in J_2 \right\rangle$. Furthermore, let $\alpha: \C^n \to \C^n/(V_1 \oplus V_2)$ be the projection. 

Now we define 
\begin{align*}
	\vp: \st w &\to \com S \\
	(F_1 \subset \dots \subset F_n) & \mapsto \big(\alpha(F_{j_1}) \subset \dots \subset \alpha(F_{j_r})\big)
\end{align*}
where $\{j_1, \dots, j_r\} = \{1, \dots, n\} \setminus J$.

Then we have $\alpha(F_{j_r}) = \alpha(F_n) = \C^n/(V_1 \oplus V_2) = \C^{n'}$.

Let $N': \C^{n'} \to \C^{n'}$ be defined by $N' e'_i = e'_{i-1}$ and $N' f'_j = f'_{j-1}$ for $i,j = 1, \dots, \frac{n'}2$, where 
	\[\{e'_1, \dots , e'_{n'}\} = \{e_1 + (V_1 \oplus V_2), \dots, e_{n-k}+ (V_1 \oplus V_2)\} \setminus \{e_j + (V_1 \oplus V_2) | e_j \in V_1\}
\]
 and the same for the $f'_j$'s. 

Since $(F_1 \subset \dots \subset F_n)$ is $N$-invariant and satisfies the conditions of $\depG{w}$ and $N'$ acts in the same way as $N$ acts on the subspaces associated to black arcs, $\big(\alpha(F_{j_1}) \subset \dots \subset \alpha(F_{j_r})\big)$ is $N'$-invariant and satisfies the conditions of $\depG{S}$.

Furthermore, the map given above is a morphism of varieties by an argumentation analogous to the one in the proof of Theorem \ref{Abb nach Springer}.

Analogously, we can construct an inverse map by defining the subspaces associated to a green cup as given by $\depG{w}$, which is also a morphism of varieties. 
\end{bew}

\begin{bew}[Proof of Theorem \ref{bundles}] 
First we assume $n=2k$ and show the theorem for standard tableaux $S$ of type $(1, \dots, n)$, i.e. for irreducible components in the Springer fibre.
We do induction on $n$:\\
If $k=2$, then by Example \ref{special cases} in both cases we get $2$-bundles by including $\CP^1 \stackrel{\id}{\rightarrow} \CP^1$ where the fibre is just one point.

Now, we consider $\com S$ with $|eC(S)| =2k$.
If $eC(S)$ is in nested position, then by Lemma \ref{case2} $\com S$ is the total space of a fibre bundle with base space $\CP^1$ and fibre $\com T$ with $|eC(T)| = 2k-2$, where $eC(T)$ arises from the cup diagram decomposition. By induction $\com T$ is a $k-1$-bundle, thus $Y_S$ is a $k$-bundle. 

In the other case, by Lemma \ref{case1} $\com S$ is the total space of a trivial fibre bundle with base space $Y_R$ and fibre $Y_T$, where $|eC(R)|, |eC(T)| \leq 2k-2$ and $|eC(R)| + |eC(T)| = 2k$. By induction, $Y_R$ and $Y_T$ are $\frac{|eC(R)|}2-$ or $\frac{|eC(T)|}2$-bundles, respectively. Now the assertion follows from Lemma \ref{trivialfib}.

By the above, for $S$ a standard tableau of type $(1, \dots, n)$ and $n=2k$, $Y_S$ is an iterated fibre bundle with $\CP^1$ as base spaces, hence it is smooth as variety by \cite[III.9,10]{hartshorne}.
Hence, by Proposition \ref{löschegrüne}, for $w$ a row strict tableau of type $(1, \dots, n)$ and shape $(n-k,k)$ with $n-k > k$, $\st w$ is also an iterated fibre bundle and smooth. Thus, by Theorem \ref{Abb nach Springer}, for $w$ a row strict tableau of type $(i_1, \dots, i_m)$, $\stb w$ is an iterated fibre bundle and smooth.

So, by Theorem \ref{Abb nach Springer} we can restrict ourselves to generalised irreducible components in Springer fibres, i.e. to those associated to cup diagrams without $\times \times$.
By Proposition \ref{löschegrüne} it suffices to consider irreducible components associated to standard tableaux with equally long rows, i.e. those associated to cup diagrams with only black cups.
Therefore, by the above discussion the assertion is shown for generalised irreducible components of Spaltenstein varieties. 
The assertion about the number of independents is transferred, since by Remark \ref{unabhängige2} the number of independents is given by the number of black cups and this stays the same in the reduction process.

Now we consider intersections of generalised irreducible components. Again, we only have to consider the ones with $eC(w,w')$ just consisting of black circles. This reduction is possible, since analogously as above with Theorem \ref{Abb nach Springer} we can restrict ourselves to intersections of generalised components in Springer fibres. Then we can delete the green circles by the analogon of Proposition \ref{löschegrüne}. These analogons can be shown by computations similar to the ones above. 

After that, we distinguish whether there is a circle in $eC(w,w')$ which contains all the others or not, and work with a decomposition of the circle diagrams analogously to the one of the cup diagrams above. In the case where there is a circle containing all the others, we get an analogon to Lemma \ref{case2} and in the other case one to Lemma \ref{case1}. Then the assertion follows inductively as above. The assertion about the number of independents is also true, because by Remark \ref{unabhängige3} the number of independents is given by the number of black circles.
\end{bew}

\section{Consequences and cohomology} 
\label{sec:newSection}

\begin{bem}
In the proof of Theorem \ref{bundles} we showed that the $\stb w$ are smooth as varieties, hence $(\stb w)^{an}$ or $(\stb w \cap \stb {w'})^{an}$ are complex manifolds.

Since the dimensions of manifolds in a fibre bundle add up, by Theorem \ref{bundles} we get that the dimension of $(\stb w)^{an}$ or $(\stb w \cap \stb {w'})^{an}$ is given by the number of independents in the associated dependence graph. Since by \cite{gaga}
the dimension of $X$ as variety coincides with the dimension of $X^{an}$ as complex manifold, the same holds for $\stb w$ and $\stb w \cap \stb {w'}$.
For $\stb w$ by Remark \ref{unabhängige2} the dimension coincides with the number of black cups in $eC(w)$. 

In particular, this holds for $\stb S = \comb S$. Thus, all the irreducible components of the Spaltenstein variety have the same dimension. This is true, because for a standard tableau $S$ the number of black cups in $eC(S)$ is exactly the number of entries in $S_{\down} \smallsetminus S_{\times}$ which is always $k- \#S_{\times}$.
For irreducible components of the Springer fibre with possibly more than two blocks this was already shown in \cite{spal}. The equidimensionality for general Spaltenstein varieties is also mentioned in \cite{spal} and proved in wider generality in \cite[II.5.16]{spalb}. 
\end{bem}

In the following, $H^*(X;\C)$ denotes singular cohomology with coefficients in $\C$.

\begin{lemma} \label{specseq} Let $(F \to E \to B)$ be a fibre bundle with $F$ connected, $\dim_{\C} H^n(F;\C)$ and $\dim_{\C} H^n(B;\C)$ finite for all $n$, $B$ simply connected, paracompact and Hausdorff. Assume that $H^*(F; \C)$ and $H^*(B;\C)$ are concentrated in even degrees and $H^{r}(F;\C) = 0$ for $r \geq s$ and $H^r(B;\C) = 0$ for $r \geq t$. \\Then $H^*(E; \C) \cong H^*(B; \C) \otimes_{\C} H^*(F;\C)$ as vector spaces. Furthermore, $\dim_{\C} H^n(E;\C)$ is finite for all $n$, $H^*(E;\C)$ is concentrated in even degrees and $H^r(E;\C) = 0$ for $r \geq s+t$.
\end{lemma}
\begin{proof}
By \cite[§2.7]{spanier} a fibre bundle with paracompact and Hausdorff base space is a fibration.
The system of local coefficients is simple because $\pi_1(B) = 0$. 
Thus by \cite[5.5]{mccleary} we have $E_2^{p,q} \cong H^p(B;\C) \otimes_k H^q(F;\C)$ in the Leray-Serre spectral sequence. Because of the concentration of the cohomologies in even degrees, the spectral sequence collapses at level $2$ and we have $E_{\infty}^{p,q} \cong H^p(B;\C) \otimes_{\C} H^q(F;\C)$. Since the spectral sequence converges to $H^*(E;\C)$ and ${\C}$ is a field, we have 
	\[H^l(E;\C) \cong \bigoplus_{a+b=l} E_{\infty}^{a,b} \cong \bigoplus_{a+b=l} H^a(B;\C) \otimes_{\C} H^b(F;\C).
\]
 The rest follows from $H^*(E;\C) = \bigoplus_l H^{l}(E; \C)$. 
\end{proof}

In the following, we write $H^*(X)$ for $H^*(X^{an}; \C)$. 

\begin{kor}\label{cohomology}
Let $w, w'$ be row strict tableaux of type $(i_1, \dots, i_m)$ and assume $\stb w \cap \stb {w'} \neq \emptyset$. Then
\begin{align*}
	H^*(\stb w) \cong \left(\C[x]/(x^2)\right)^{\otimes u} \\
	H^*(\stb w \cap \stb {w'}) \cong \left(\C[x]/(x^2)\right)^{\otimes v}
\end{align*}
as vector spaces, where $u$ and $v$ are the number of independents.
\end{kor}
\begin{bew}
By Theorem \ref{bundles} we know that $(\stb w)^{an}$ and $(\stb w \cap \stb {w'})^{an}$ are $l$-bundles. In each of the iterated fibre bundles the fibre is connected, since the fibres themselves are iterated fibre bundles.

$\CP^1$ is simply connected, paracompact and Hausdorff and $H^*(\CP^1; \C) \cong \C[x]/(x^2)$, in particular $H^0(\CP^1; \C) = \C = H^2(\CP^1; \C)$ and $H^l(\CP^1; \C) = 0$ for $l=1$ or $l \geq 3$. 

Now, the claim inductively follows from Lemma \ref{specseq} starting with $H^*(\pt) = \C$, since the lemma also states that the conditions for the next step are fulfilled. 
Therefore, for the total space $E$ of an $l$-bundle we have 
\begin{align*}
	H^*(E; \C) \cong \underbrace{\C[x]/(x^2)\otimes_{\C} \dots \otimes_{\C} \C[x]/(x^2)}_{l} \otimes_{\C} \C \cong \left(\C[x]/(x^2)\right)^{\otimes l}.
\end{align*}
\end{bew}

\begin{defi}\label{F}
Define $F: \{\text{circle diagrams}\} \to \{\text{vector spaces over }\C \}$ as follows:\\
Let $C$ be a circle diagram consisting of $b$ black circles, $g$ green circles and $r$ red circles. Then 
	\[F(C) := \underbrace{\C[x]/(x^2)\otimes_{\C} \dots \otimes_{\C} \C[x]/(x^2)}_b \otimes_{\C} \underbrace{\C\otimes_{\C} \dots \otimes_{\C} \C}_g \otimes_{\C} \underbrace{0 \otimes_{\C} \dots \otimes_{\C} 0}_r.
\]
\end{defi}
The concept of this definition will be made clear using coloured TQFT's in the next section.

\begin{satz}
\label{toPartI}  ~\\ 
	The following diagram commutes:
	\[\xymatrix{
	\text{pairs of row strict tableaux}  \ar[rrrr]_{\quad \quad \quad (w,w') \mapsto F(CC(w,w'))} \ar[dd] &&&&  \text{vector spaces} \ar[dd]^{\id}\\
	  \\
	 \text{pairs of }\stb w \ar[rrrr]_{(\stb w, \stb {w'}) \mapsto H^*(\stb w \cap \stb {w'}) } &&&&  \text{vector spaces}\\
	 }\]
		
\end{satz}
\begin{bew}
This follows from Theorem \ref{rote}, Remark \ref{unabhängige3} and Corollary \ref{cohomology}.
\end{bew}

\section{Coloured Cobordisms}
\label{sec:neuerAbschnitt}

Let $\Cob$ be the category of two-dimensional cobordisms. By \cite{kock} 
this monoidal category is generated under composition and disjoint union by the cobordisms 	\[\comult, \mult, \unit, \counit, \cobid, \twist.
\] Furthermore, these generators are subject to an explicit list of relations (see. e.g. \cite{kock}) saying that the image of the circle under a symmetric monoidal functor should be a commutative Frobenius algebra.

Now we consider coloured cobordisms, i.e. the boundaries of the generators are coloured black, green or red.

\begin{defi} \label{ColCob} Let $\ColCob$ be the monoidal category generated under composition and disjoint union by 
\begin{center}
~\\[-4.7ex] 
{\setlength{\tabcolsep}{-1pt}
\begin{tabular}{cccccccccc}
  {\boxenlenb{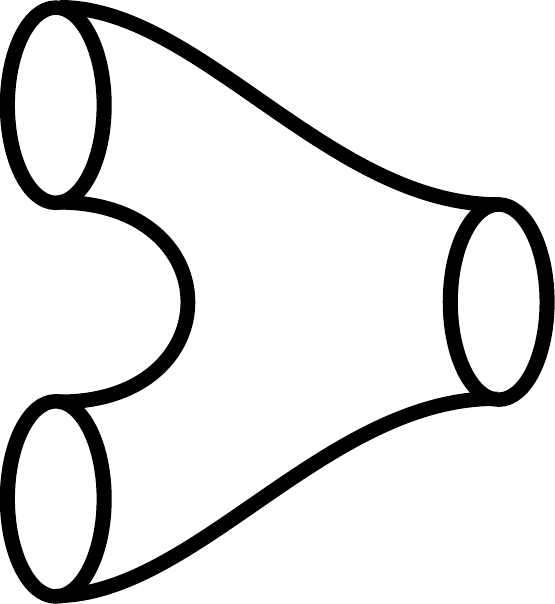}{6}}
	&{\boxenlenb{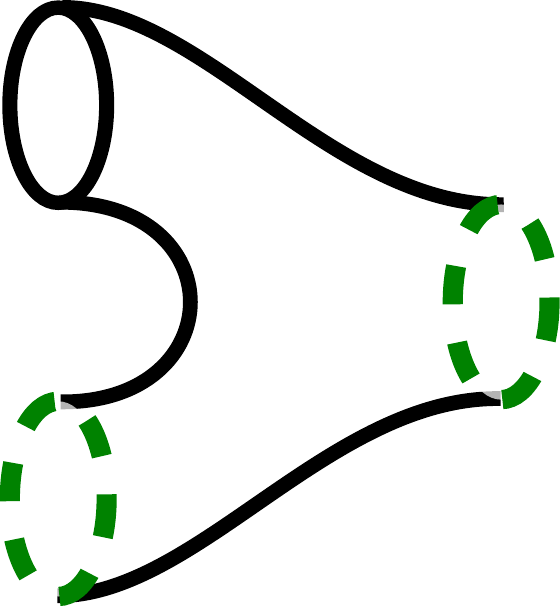}{6}}%
	&{\boxenlenb{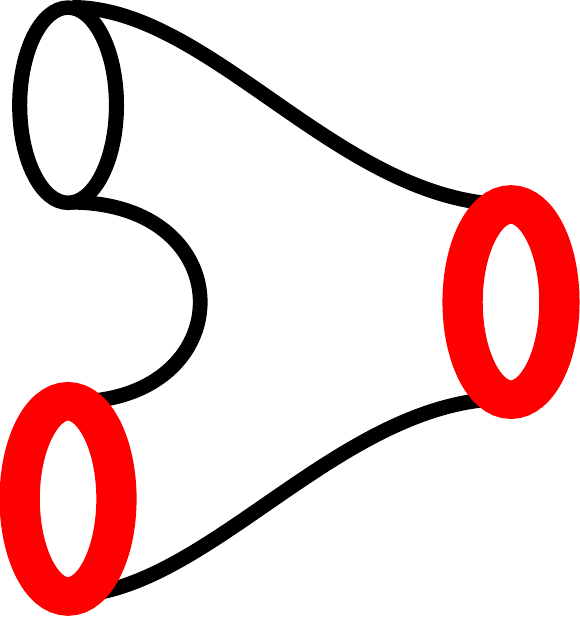}{6}}%
	&{\boxenlenbx{multsg}{6}}
	&{\boxenlenb{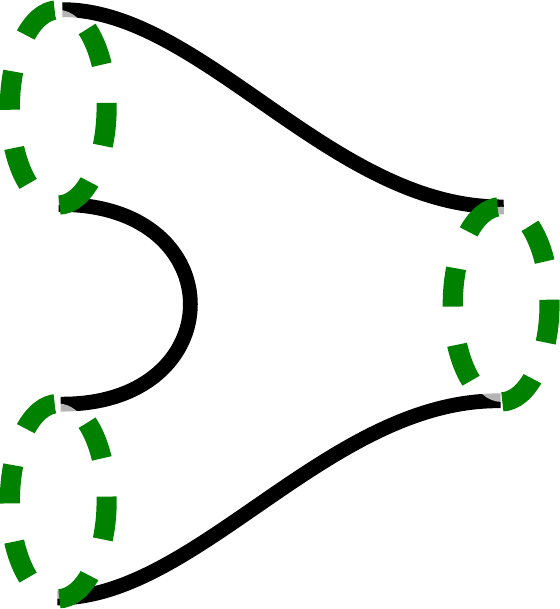}{6}}%
	&{\boxenlenb{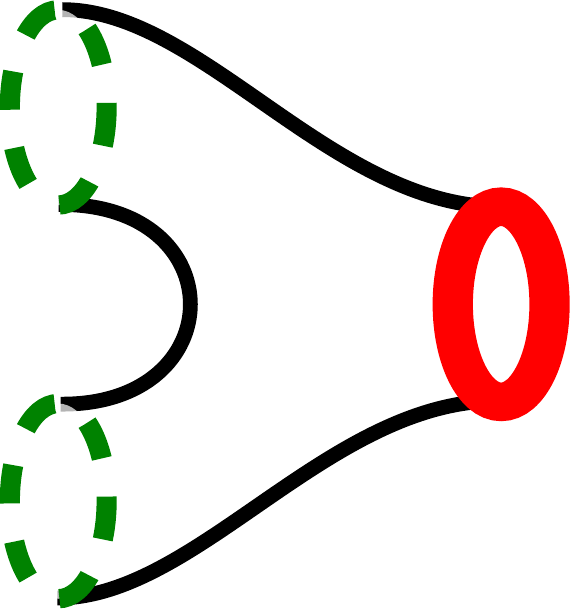}{6}}%
	&{\boxenlenb{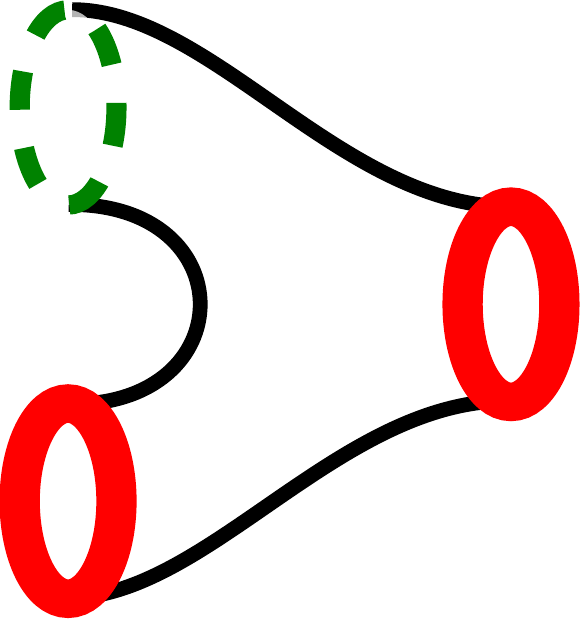}{6}}%
	&{\boxenlenbx{multsr}{6}}
		&{\boxenlenbx{multgr}{6}}
			&{\boxenlenb{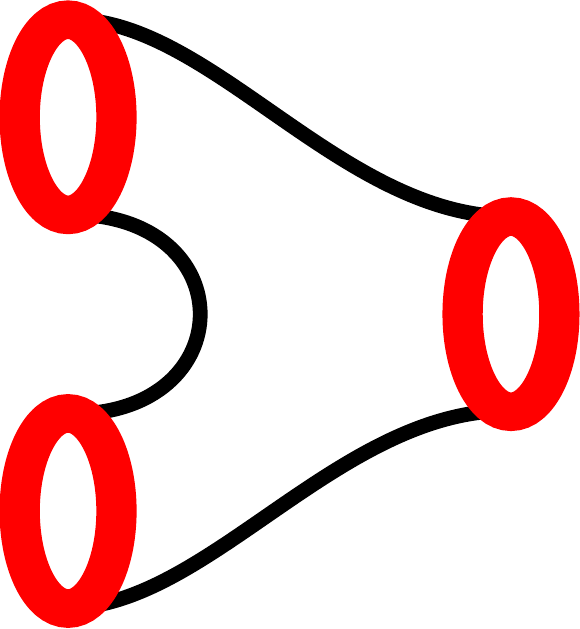}{6}} \\[2ex]
			{\boxenlenby{multss}{6}}
				&{\boxenlenby{multsg}{6}}
				&{\boxenlenby{multsr}{6}}
				&{\boxenlenbd{multsg}{6}}
				&{\boxenlenby{multgg}{6}} 
				&{\boxenlenby{multgg2}{6}}
				&{\boxenlenby{multgr}{6}}
				&{\boxenlenbd{multsr}{6}}
				&{\boxenlenbd{multgr}{6}}
				&{\boxenlenby{multrr}{6}}\\[2ex]
				{\boxenlenb{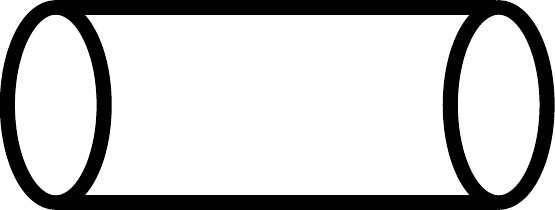}{2.5}} 
				&{\boxenlenb{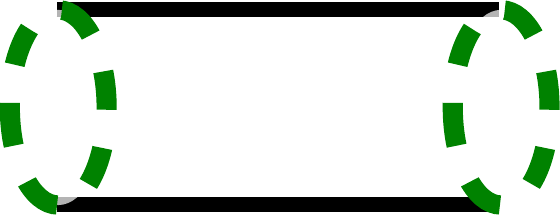}{2.5}}
				&{\boxenlenb{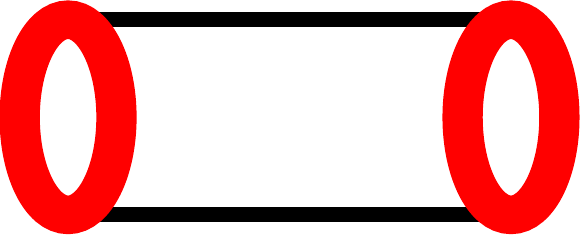}{2.5}}
				&{\boxenlenb{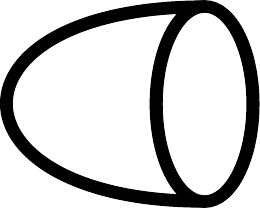}{3}}
				&{\boxenlenby{units}{3}}\\[2ex]
				{\boxenlenb{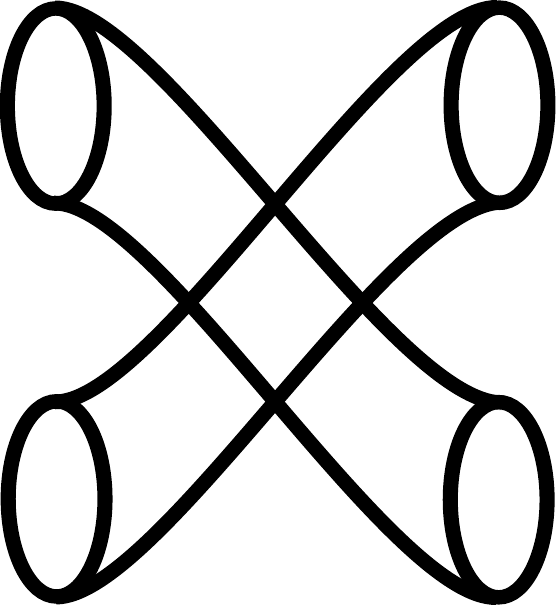}{6}}
				&{\boxenlenb{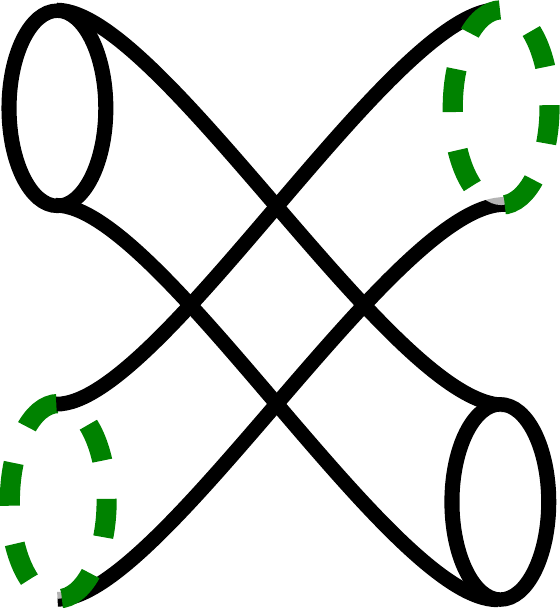}{6}}
				&{\boxenlenb{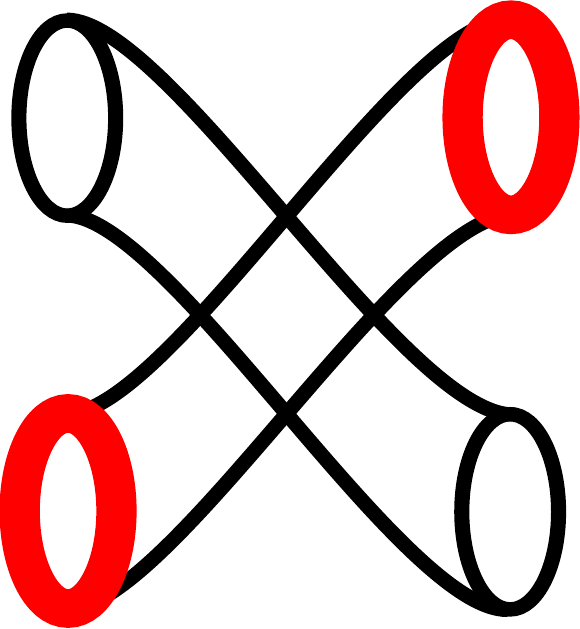}{6}}
				&{\boxenlenby{twistsg}{6}}
				&{\boxenlenb{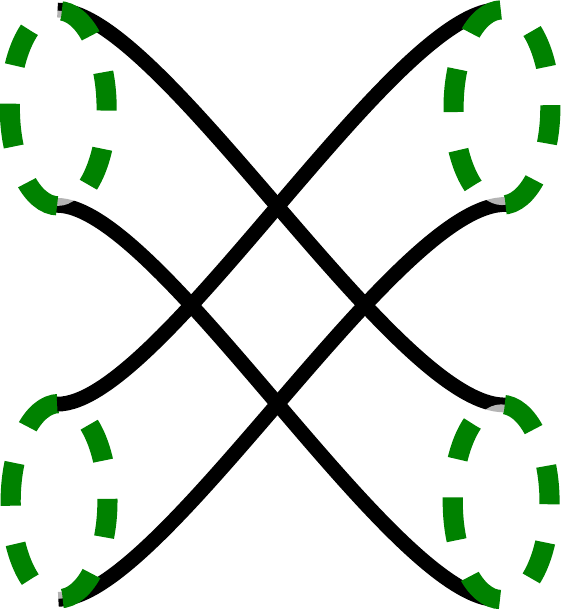}{6}}
				&{\boxenlenb{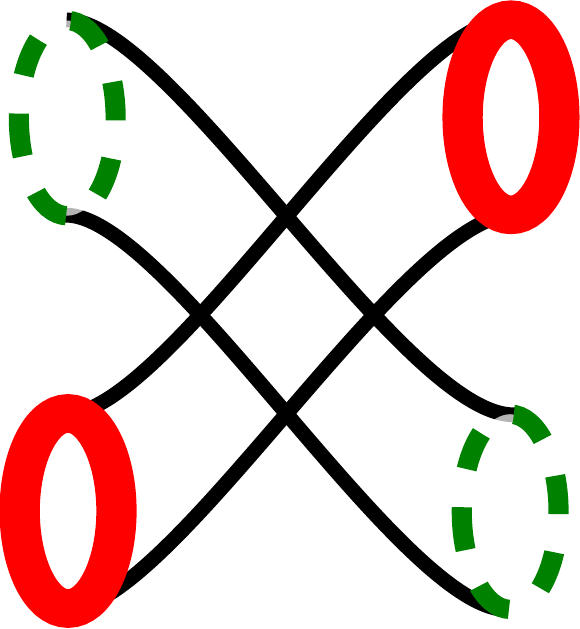}{6}}
				&{\boxenlenby{twistsr}{6}}
				&{\boxenlenby{twistgr}{6}}
				&{\boxenlenb{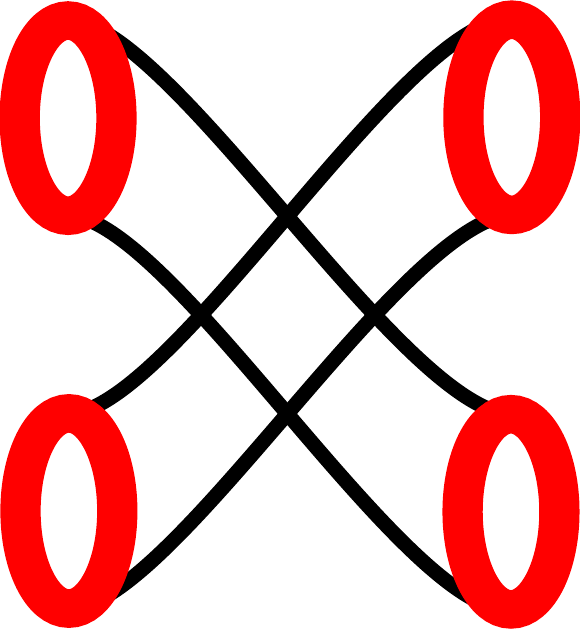}{6}}
\end{tabular}}\\[1.2ex]
\end{center}
subject to the relations for $\ColCob$. 
The relations for $\ColCob$ consist of all the fitting colourings of the relations for $\Cob$, i.e. if the cobordisms of a relation can both be coloured such that the basic cobordisms they consist of are in the list and the boundaries are coloured in the same way, then the relation exists in this colouring. For an explicit list of the relations see \cite[A.2]{me}.
\end{defi}

\begin{bsp}
An example for an object in $\ColCob$: 
\begin{tabular}[c]{c}
\begin{tikzpicture}
\draw[ultra thick, red] (0,0) circle (2mm);
\draw[thick] (0,0.5) circle (2mm);
\draw[dashed,thick, green!50!black] (0,1) circle (2mm);
\draw[thick] (0,1.5) circle (2mm);
\end{tikzpicture}
\end{tabular}\\
An example for a morphism in $\ColCob$:
\boxenlen{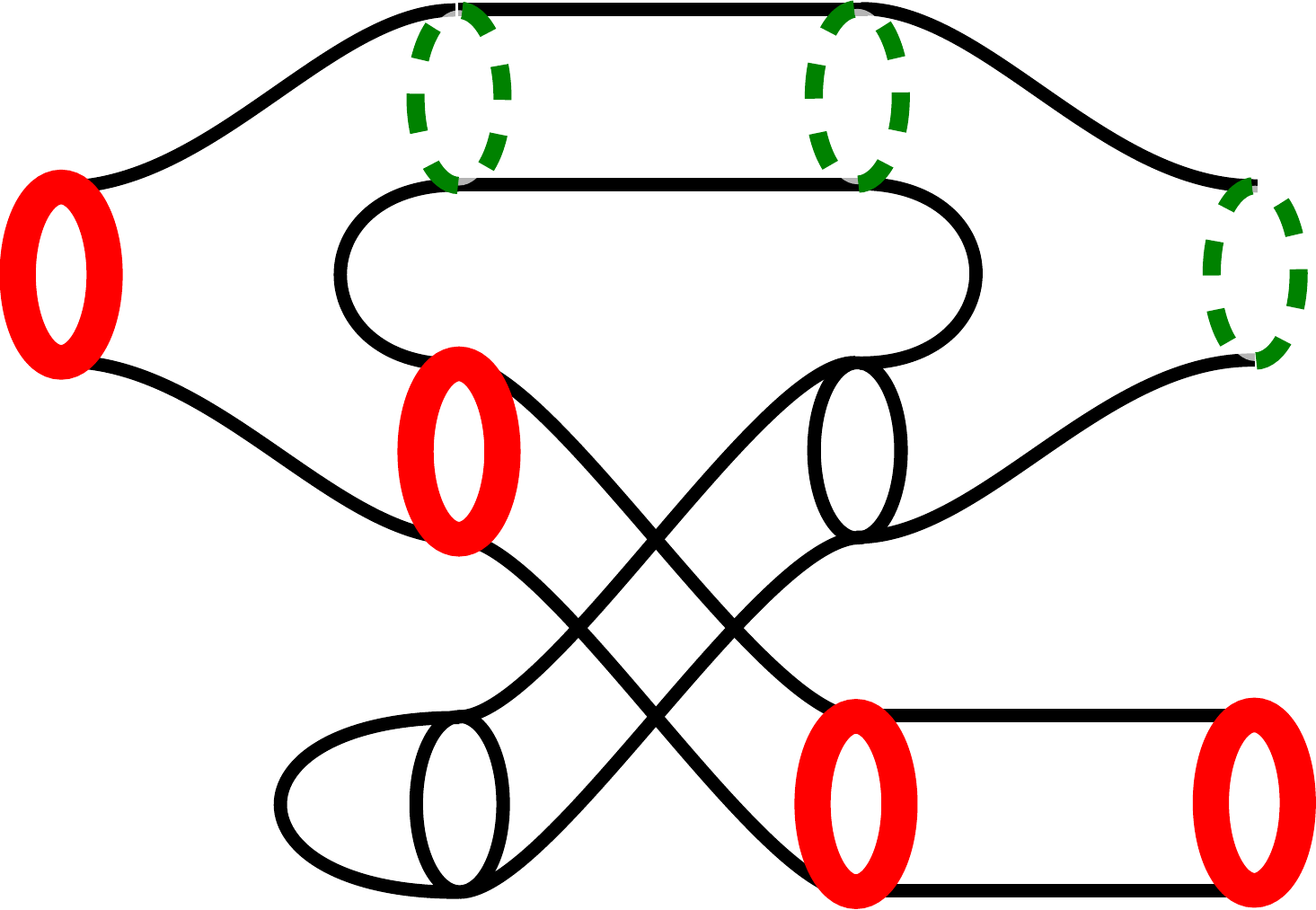}{2}
\end{bsp}

\begin{lemma}
$\ColCob$ is a symmetric monoidal category.
\end{lemma} 
\begin{bew}
This holds analogously to $\Cob$, since the twist relations exist in all possible colourings.
\end{bew}

\begin{satz} \label{functorColCob}
Let $V= \mathbb{C}[x]/(x^2)$, let $B$ be the black circle in $\ColCob$, $R$ the red one and $G$ the green one and let $\op{Vect}$ be the monoidal category of vector spaces with ordinary tensor product. There exists a symmetric monoidal functor $F_C=F_{\ColCob}: \ColCob \to \op{Vect}$ given by 
\begin{align} \label{star}
\begin{array}{rl}
	F_C(B) = V,\quad F_C(G) &= \mathbb{C},\quad F_C(R) = 0, \\
	F_C(\text{generator}) &= \text{map from table below}. 
	\end{array}
\end{align}

\begin{center}
	\begin{tabular}{|l|l|}
	\hline
		{\boxenlenb{multss}{5}} & \parbox{3cm}{~\\[-0.5ex]$1 \otimes 1 \mapsto 1$\\ 
					 $x \otimes 1 \mapsto x$\\
					 $1 \otimes x \mapsto x$\\
					 $x \otimes x \mapsto 0$\\[-2.3ex]}\\
		\hline
		{\boxenlenb{multsg}{5}} & \parbox{3cm}{~\\[-0.5ex]$1 \otimes 1 \mapsto 1$\\ 
					 $x \otimes 1 \mapsto 0$\\[-2.3ex]}\\
			\hline
		{\boxenlenb{multsr}{5}} & \parbox{3cm}{~\\[-0.5ex] $1 \otimes 0 \mapsto 0$\\ 
					 $x \otimes 0 \mapsto 0$\\[-2.3ex]}\\
		\hline
		{\boxenlenbx{multsg}{5}} & \parbox{3cm}{~\\[-0.5ex] $1 \otimes 1 \mapsto 1$\\ 
					 $1 \otimes x \mapsto 0$\\[-2.3ex]}\\
		\hline
		{\boxenlenb{multgg}{5}} & $1 \otimes 1 \mapsto 1$\\ 
			\hline
		{\boxenlenb{multgg2}{5}} & $1 \otimes 1 \mapsto 0$\\ 
				\hline
		{\boxenlenb{multgr}{5}} & $1 \otimes 0 \mapsto 0$\\ 
		\hline
		{\boxenlenbx{multsr}{5}} & \parbox{3cm}{~\\[-0.5ex] $0 \otimes 1 \mapsto 0$\\ 
					 $0 \otimes x \mapsto 0$\\[-2.3ex]}\\
			\hline
		{\boxenlenbx{multgr}{5}} & $0 \otimes 1 \mapsto 0$\\ 
			\hline
		{\boxenlenb{multrr}{5}} & $0 \otimes 0 \mapsto 0$\\ 
		\hline
		{\boxenlenby{multss}{5}} & \parbox{3cm}{~\\[-0.5ex] $1 \mapsto x \otimes 1 + 1 \otimes x$\\
						 $x \mapsto x \otimes x$\\[-2.3ex]}\\
				\hline
				\end{tabular}
				\hspace{2cm}
		\begin{tabular}{|l|l|}
				\hline
		{\boxenlenby{multsg}{5}} & $1 \mapsto x \otimes 1 $\\
		\hline
		{\boxenlenby{multsr}{5}} & $0 \mapsto 0 \otimes 0$\\
		\hline
		{\boxenlenbd{multsg}{5}} & $1 \mapsto 1 \otimes x$\\
				\hline
		{\boxenlenby{multgg}{5}} & $1 \mapsto 0 \otimes 0 $\\
		\hline
		{\boxenlenby{multgg2}{5}} & $0 \mapsto 0 \otimes 0$\\
		\hline
		{\boxenlenby{multgr}{5}} & $0 \mapsto 0 \otimes 0$\\
		\hline
		{\boxenlenbd{multsr}{5}} & $0 \mapsto 0 \otimes 0$\\
		\hline
		{\boxenlenbd{multgr}{5}} & $0 \mapsto 0 \otimes 0$\\
		\hline
		{\boxenlenby{multrr}{5}} & $0 \mapsto 0 \otimes 0$\\
		\hline
		{\boxenlenb{ids}{2}} & \parbox{3cm}{~\\[-0.5ex]$1 \mapsto 1$\\
					  $x \mapsto x$\\[-2.3ex]}\\
		\hline
		{\boxenlenb{idg}{2}} & $1 \mapsto 1$\\
		\hline
		{\boxenlenb{idr}{2}} & $0 \mapsto 0$\\
		\hline
		{\boxenlenb{units}{3}}  & $1 \mapsto 1$\\
		\hline
		{\boxenlenby{units}{3}} &\parbox{3cm}{~\\[-0.5ex] $1 \mapsto 0$\\
						 $x \mapsto 1$\\[-2.3ex]}\\
		\hline
		\text{Twists} & $a \otimes b \mapsto b \otimes a$\\
		\hline
	\end{tabular}
\end{center}
\end{satz}
\begin{bew}
For $F_C$ to be a monoidal functor it is enough to define it on the generators of the monoidal category, hence there is a unique functor (if it exists) satisfying \eqref{star}. One has to check that $F_C$ satisfies all the relations of $\ColCob$ (which is the list of \cite{kock} in all possible colourings), for example: 
\begin{align*}
		&F_C\left(\boxenleny{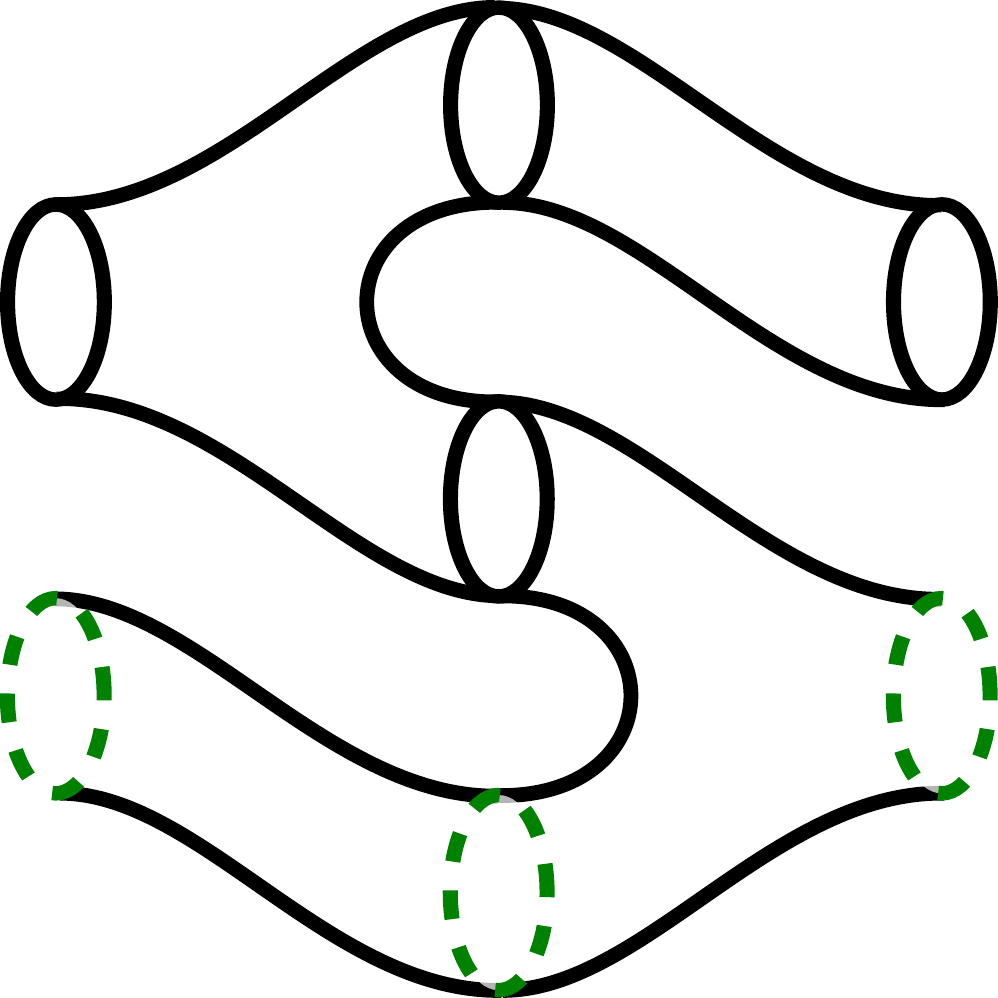}{1.25}\right)  = \left(\id_S \otimes F_C\left( \boxenlenby{multsg}{5}\right) \right)\circ \left(F_C\left( \boxenlenb{multss}{5}\right)\otimes \id_G\right)\\
	&=\left(\begin{array}{ccccc} 
	1 \otimes 1 &\mapsto&  1 \otimes x \otimes 1&\mapsto & x \otimes 1\\
	x \otimes 1 &\mapsto&  x \otimes x \otimes 1 & \mapsto & 0\\
	\end{array} \right)
\end{align*}
\begin{align*}
		&F_C\left(\boxenlen{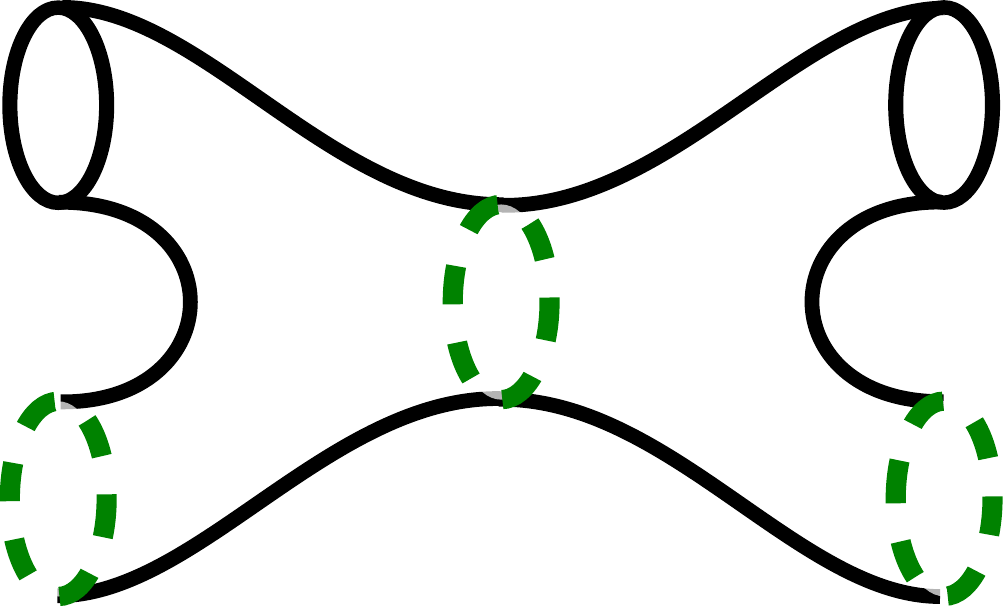}{0.75}\right)  = F_C\left( \boxenlenb{multsg}{5}\right) \circ F_C\left(\boxenlenby{multsg}{5} \right)\\
	&=\left(\begin{array}{ccccc} 
	1 \otimes 1 &\mapsto& 1 &\mapsto &x \otimes 1\\
	x \otimes 1 &\mapsto& 0  & \mapsto & 0\\
	\end{array} \right)
\end{align*}
The functor is symmetric since the twist in $\ColCob$ is sent to the twist in $\op{Vect}$.
\end{bew}

Note that on objects $F_{\ColCob}$ is defined in the same way as the $F$ in Definition \ref{F}.

\begin{defi}\label{Colmult}
For a circle diagram $CC(w,w')$ where the points $-1,-2, \dots, -(n-k)+|w_{\times}|$ are not all occupied, we add green circles containing all the others until this is the case. We call the resulting circle diagram $CC^+(w,w')$. 

Note that $F_C(CC(w,w')) = F_C(CC^+(w,w'))$ and now all circle diagrams associated to row strict tableaux of type $(i_1, \dots, i_m)$ have the same size.

We define a multiplication as in \cite[5.4]{stroppel2} via $F_C(CC^+(w,w')) \otimes F_C(CC^+(v,v')) \to F_C(CC^+(w,v'))$, $f \otimes g \to fg$ where $fg= 0$ if $w'\neq v$ and otherwise $fg= F_C(C)$ where $C$ is the coloured cobordism from 
$CC(w',v)$ on top of $CC(w,w')$ to $CC(w,v)$ which contracts the parts belonging to $w'$.
\end{defi}

\begin{satz} \label{algebra}
The multiplication from Definition \ref{Colmult} induces an associative algebra structure on $\bigoplus_{w,w'}H^*(\stb w \cap \stb{w'})$. 
\end{satz}
\begin{bew}
The multiplication is associative, since the coloured associativity relations hold in $\ColCob$.
\end{bew}

It would be interesting to find an 
algebraic formulation of the data of a symmetric monoidal functor from $\ColCob$ to 
finite dimensional vector spaces
(in analogy to the case of $\Cob$ where such a functor can be described equivalently by the structure of a commutative Frobenius algebra).

\begin{bem}(Connection to category $\EuScript{O}$)\\ 
Note that the algebra structure for Spaltenstein varieties is the same as for Springer fibres of smaller dimension. A similar phenomenon arises in Lie theory in the context of parabolic category $\EuScript{O}$. By the Enright-Shelton equivalence singular blocks of parabolic category $\EuScript{O}$ for $\mathfrak{gl}_{m+n}$ are equivalent to regular blocks for smaller $m$ and $n$ (see \cite[Proposition 11.2]{ES}). In fact, in \cite{BSIII} this analogy was made precise by constructing an equivalence between modules over our diagram algebras and blocks of parabolic category $\EuScript{O}$.

In our setup, the Springer fibre corresponds to the principal block $\EuScript{O}_0^{\mathfrak{p}_{n-k,k}}$ of parabolic category $\EuScript{O}$ for the Lie algebra $\mathfrak{gl}_n$ where the parabolic has two blocks of size $n-k$ and $k$, respectively (cf. \cite{stroppel}). 
A Spaltenstein variety of type $(i_1, \dots, i_m)$ correspond to a block $\EuScript{O}_{\nu}^{\mathfrak{p}_{n-k,k}}$, where $\nu = (i_1, i_2-i_1, \dots, i_m-i_{m-1})$ is the corresponding partition of $n$. 
In this block, the simple modules are labelled by row strict tableaux of shape $(k, n-k)$ of type $(i_1, \dots, i_m)$ (cf. \cite{brundan}) as are the generalised irreducible components of $Sp(i_1, \dots, i_m)$. 

Whereas Theorem \ref{Abb nach Springer} can be used to reduce from Spaltenstein varieties to smaller Springer fibres, the counterpart on the category $\EuScript{O}$ side is the (non-trivial) Enright-Shelton equivalence. 
Diagrammatically, this equivalence is obvious. 
\end{bem}

%%%%%%%%%%%%%%%%%%%%%%%%%%%%%%%%%%%%%%%%%%%%%%%%%%%%%%%%%%%%%%%%%%%%%%%%%%%%%%%%%%%%%%%%%%%%%%%%%%%%%%%%%%%%%%%%%%%%%
%\providecommand{\bysame}{\leavevmode\hbox to3em{\hrulefill}\thinspace}
%\providecommand{\MR}{\relax\ifhmode\unskip\space\fi MR }
%% \MRhref is called by the amsart/book/proc definition of \MR.
%\providecommand{\MRhref}[2]{%
%  \href{http://www.ams.org/mathscinet-getitem?mr=#1}{#2}
%}
%\providecommand{\href}[2]{#2}

\end{document}